\documentclass[11pt, letterpaper]{amsart}

\usepackage{mathpazo, eulervm}
\usepackage[english]{babel}
\usepackage[utf8]{inputenc}
\usepackage[T1]{fontenc}
\usepackage{graphicx, amssymb, amsmath, accents, amsthm, bm}
\usepackage{geometry}
\usepackage{tikz-cd}
\usepackage[numbers]{natbib}

\newtheorem{definition}{Definition}
\newtheorem{lemma}{Lemma}
\newtheorem{theorem}{Theorem}

\newtheorem{corollary}{Corollary}
\theoremstyle{remark}

\newenvironment{acknowledgements}{\subsection*{Acknowledgements}}{}

\DeclareMathOperator*{\argmin}{arg\,min}
\newcommand{\X}{\mathbb{X}}
\newcommand{\Y}{\mathbb{Y}}
\newcommand{\R}{\mathbb{R}}
\newcommand{\Z}{\mathbb{Z}}

\newcommand{\st}{\mathcal{S}}

\newcommand{\wfs}{\mathrm{wfs}}
\newcommand{\e}{\varepsilon}

\newcommand{\im}{\mathrm{im\,}}
\newcommand{\ii}{i_*}
\newcommand{\jj}{j_*}
\newcommand{\h}{H}
\newcommand{\hl}{H_{\ell}}

\newcommand{\vphi}{\varphi}
\newcommand{\vrho}{\upsilon}
\newcommand{\reach}{\mathrm{reach}}
\newcommand{\card}{\mathrm{card}\,}
\newcommand{\cl}[1]{\mathrm{cl}\left(#1\right)}
\newcommand{\intr}[1]{\mathrm{int}\left(#1\right)}

\newcommand{\height}[1]{\mathrm{ht}{(#1)}}

\newcommand{\cech}{\mathcal{C}}
\newcommand{\rips}{\mathcal{R}}
\newcommand{\conv}{\mathrm{conv}}
\newcommand{\eqb}{\begin{equation}}
\newcommand{\eqe}{\end{equation}}
\newcommand{\rr}{\bar{r}}
\newcommand{\RR}{\bar{R}}

\newcommand{\dd}{\bar{\delta}}

\newcommand{\ignore}[1]{}

\newcommand{\Hawaii}{Hawai\kern.05em`\kern.05em\relax i }
\newcommand{\Manoa}{M\=anoa }

\begin{document}
\title{Another look at recovering local homology from
samples of stratified sets}

\author[Y. Mileyko]{Yuriy Mileyko}
\address{Department of Mathematics, University of \Hawaii at \Manoa,
Honolulu HI, USA}
\email{yury@math.hawaii.edu}

\date{}
\maketitle
\begin{abstract}
Recovering homological features of spaces from samples has become one of the
central themes of topological data analysis, leading to many successful
applications. Many of the results in this area focus on global homological
features of a subset of a Euclidean space. In this case, homology recovery
predicates on imposing well understood geometric conditions on the underlying
set. Typically, these conditions guarantee that small enough neighborhoods of
the set have the same homology as the set itself.  Existing work on recovering
local homological features of a space from samples employs similar conditions
locally. However, such local geometric conditions may vary from point to point
and can potentially degenerate. For instance, the size of local homology
preserving neighborhoods across all points of interest may not be bounded away
from zero. In this paper, we introduce more general and robust conditions for
local homology recovery and show that tame homology stratified sets, including
Whitney stratified sets, satisfy these conditions away from strata boundaries,
thus obtaining control over the regions where local homology recovery may not
be feasible.  Moreover, we show that true local homology of such sets can be
computed from good enough samples using Vietoris-Rips complexes.
\end{abstract}

\section{Introduction}
\label{sec:intro}
Estimating topological features of a space from samples is one of the
central topics in topological data analysis (TDA), which is a new field
that has been steadily gaining popularity due to a series of successful
applications
\citep[see e.g.][]{ghrist2008barcodes,carlsson2008local,de2007coverage,chan2013topology,horak2009persistent}.
The importance of such estimates stems from the
fact that they provide us with a better insight into the process
underlying the data, and can potentially help us select a better class
of generative models. Much of the work within TDA focuses on developing
and performing theoretical analyses of various methods for summarizing
global homological properties of data sets. In particular, by
imposing well understood geometric conditions on the underlying space,
several guarantees for recovery of correct homology from sufficiently
dense samples have been obtained
\citep[e.g.][]{niyogi.etal2008,niyogi2011topological,cohen-steiner.etal2007,chazal_oudot2008}.

Of course, one can easily make an argument that recovering global
homological information may not be enough. Indeed, a space having the
shape of the letter X is contractible, and thus has trivial homology,
but the presence of a singular point may be extremely important. Many of
such singular points can be captured through local homology, which
suggests that collections containing local homological features across
all the points in a sample may provide valuable information about the
underlying space. Consequently, the need arises for theoretical results
regarding such collections of local homological features. It is reasonable to expect
that such results would require certain regularity conditions on the
underlying space, and ``nice'' stratified spaces arise as a
natural class of spaces that may possess the needed properties.

There have been several impressive efforts regarding recovery of local
homological features of subsets of a Euclidean space from samples.  Such
results rely on the fact that if a set of interest is sufficiently nice,
then the local homology is ``well behaved''. The latter typically means
that for all sufficiently small $\rho>0$ and all $\e>0$ which are
sufficiently smaller than $\rho$, the following holds: the homology of
the $\e$-neighborhood of the whole set relative to the part of the
neighborhood outside of the ball of radius $\rho$ centered at a point of
interest is isomorphic to the local homology at that point.  In such a
case, we say that the set has a positive local homological feature size
at the point of interest.  The size, $\e$, of the neighborhood of our
set, as well as the radius, $\rho$, of the ball, are typically referred
to as scales. By changing $\e$ and $\rho$ we obtain nested neighborhoods
and balls along with the corresponding inclusion maps. The behavior of
the induced homomorphisms on relative homology (with coefficients in a
field) is typically summarized using persistent homology theory
\citep{edelsbrunner2002,zomorodian2005}, and may be referred to as
multi-scale local homology. The true local homology is typically
recovered from the multi-scale local homology by selecting appropriate
scales.

The work in
\cite{bendich.etal2007} focused on obtaining a multi-scale representation of local
homology at a single point of a topologically stratified set. In
particular, it was shown that the correct local homology at a single
point can be inferred from a sufficiently good sample if
the set has a positive local homological feature size at the point of
interest. Later, the work in \cite{bendich.etal2010} described a local homology based
method for assigning points of a noisy sample from a stratified set to
their respective strata. It is important to note that theoretical
guarantees for the correctness of such an assignment are based on local
conditions, akin to the positive local homological feature size, imposed around pairs of points.
It is also important to mention that the two previous results use Delaunay complexes for homology
computations. These simplicial complexes have nice theoretical properties, but are
computationally efficient only in low dimensions.
The result in \cite{skraba.wang2014} focused on an efficient approximation of a multi-scale
representation of local homology using Vietoris-Rips complexes. However,
the authors of the latter paper do not address the question of when such an
approximation captures the true local homology.

What one can take away from the above results is that the ability to
recover local homology at a point relies on local geometric conditions
(e.g. positivity of local homological feature size), which essentially
determine the appropriate range of scales for our computations.
Importantly, these conditions vary from point to point and may
``degenerate''. For example, local homological feature size may not be
bounded away from zero for a given set of points of interest, thus making the
appropriate range of scales empty. This suggests that if one were to use
the same scales to compute local homology at every point of a sample
from a stratified set, some errors may be inevitable. It is important to
be able to exercise control over the number of such errors. For a
stratum of a stratified set, degeneration of the local geometric
conditions guaranteeing local homology recovery is expected at the
boundary. However, existing result do not address the question of
whether such conditions do not degenerate away from the strata
boundaries, and consequently do not provide a way to control errors.

It should be pointed out that the problem of recovering local homology
simplifies significantly if the underlying space is a closed manifold.
In fact, the result in \cite{dey.etal2014} shows that the correct local homology of
a closed submanifold of $\R^n$ can be recovered at any point of a noisy
sample if the sample is close enough to the manifold in the Hausdorff
metric. The goal of this paper, is to obtain a somewhat analogous result
for a class of stratified sets that is large enough to subsume Whitney stratified
sets. As we mentioned earlier, the main difficulty is in obtaining some
control over the points where the recovery of the correct local homology
cannot be guaranteed, and we propose an approach that allows us to
tackle this issue. More specifically, we provide the following
contributions:
\begin{enumerate}
    \item  We introduce a concept of local homological seemliness,
        with weak, moderate, and strong flavors, which generalizes the
        typically used concept of local homological feature size.
        Roughly speaking, the main difference between the two is that for
        local homological seemliness we no longer require the relative
        homology at small scales to be isomorphic to the true local
        homology -- we only need the images of the inclusion induced
        homomorphisms between the relative homology at two sets of
        sufficiently small scales to be isomorphic to the true local
        homology. The reason for such a generalization is that Whitney
        stratified sets may fail to have positive homological feature
        size.
     \item We show how local homology can be recovered from good enough
         samples of locally homologically seemly sets using \v{C}ech and
         Vietoris-Rips complexes. As we mentioned earlier, existing
         results do not show that the true local homology can be recovered
         using Vietoris-Rips complexes.
     \item We prove that tame homology stratified sets as well as
         Whitney stratified sets are locally homologically seemly
         away from strata boundaries, thus obtaining some control over the
         region where mistakes are unavoidable. In particular, this result shows that
         with good enough samples, mistakes in recovering local homology
         of a stratum can happen only in a small region around its boundary.
     \item We show how our results strengthen when the sets under
         consideration are nicer, e.g. have positive weak feature size
         or are manifolds (with or without boundary).
\end{enumerate}
Two key results of the paper are stated in Theorems
\ref{thm:lh_recover_K} and \ref{thm:weak_seemly}. These theorems also
yield an important corollary. Suppose that $K\subset\R^n$ is a compact
neighborhood retract possessing a Whitney stratification or, more
generally, a tame homology stratification, $\st$. Let $P\subset
\R^n$ be a finite set (a noisy sample of $K$). For $Q\subseteq P$, let
$\rips_{\alpha}(Q)$ denote the Vietoris-Rips complex at scale $\alpha$
over $Q$. Denote by $d_{H}$ the Hausdorff distance on compact subsets of
$\R^n$.  The corollary can be formulated as follows.
\begin{corollary}\label{cor:intro}
    Let $\e>0$ be sufficiently small, and suppose that $d_{H}(P,K)<\e$.
    There are
    $\delta(\e)>0$, $R(\e) > r(\e)>0$, as well as strata dependent
    $w_{X}(\e)>0$, $X\in\st$, with $w_{X}(\e)\to 0$ as $\e\to 0$, such that
    for any $p\in P$, $x\in X\in\st$,
    and any homological dimensions $\ell$, the image of the inclusion
    induced homomorphism
    $$
        \hl(\rips_{\e}(P),\rips_{\e}(P-B_{R(\e)}(p))) \to
        \hl(\rips_{\delta(\e)}(P),\rips_{\delta(\e)}(P-B_{r(\e)}(p)))
    $$
    is isomorphic to the local homology 
    $
        \h(K,K-\{x\})
    $
    as long as $\|x-p\|<\e$ and $\|x-y\|\geq w_{Y}(\e)$, where $Y$ is any
    stratum in the boundary of $X$ and $y\in Y$.
\end{corollary}
More details about the relevant concepts and the values of admissible
scales are provided in subsequent sections. The corollary shows that
even for fairly general stratified sets, we can compute correct local
homology from samples for each stratum, except possibly for points close
to the boundary of the strata.  Moreover, we can do it efficiently using
Vietoris-Rips complexes.

On the practical side, such results provide a justification for using
Vietoris-Rips complexes when computing local homology, along with an
additional insight into the range of applicable scales. A simple example
is shown in Figure \ref{fig:num_example}. 
The values of the quantities
$\delta(\e)$, $r(\e)$ and $R(\e)$, as employed in Corollary
\ref{cor:intro}, were chosen based on Corollary
\ref{cor:lh_strong_simple} in Section \ref{sec:seemly} and the
estimate mentioned right after that corollary. These choices
(implicitly) determine the values of $w_{Y_1}(\e)$ and $w_{Y_2}(\e)$, and we can
see how they affect correctness of local homology estimation at the
points along the $1$-dimensional strata that are too close to the
boundary. In fact, Corollary \ref{cor:intro} says that regardless of our
choice of admissible scales, there may always be regions close to the
$0$-dimensional strata where local homology estimation will fail.
We would also like to point out that the cost of our computation would
remain essentially the same even if we isometrically embedded the given
stratified set in a very high dimensional Euclidean space. In contrast,
such an embedding would make the use of Delaunay complexes
computationally prohibitive.

\begin{figure}
    \centering
    \includegraphics[width=0.9\textwidth]{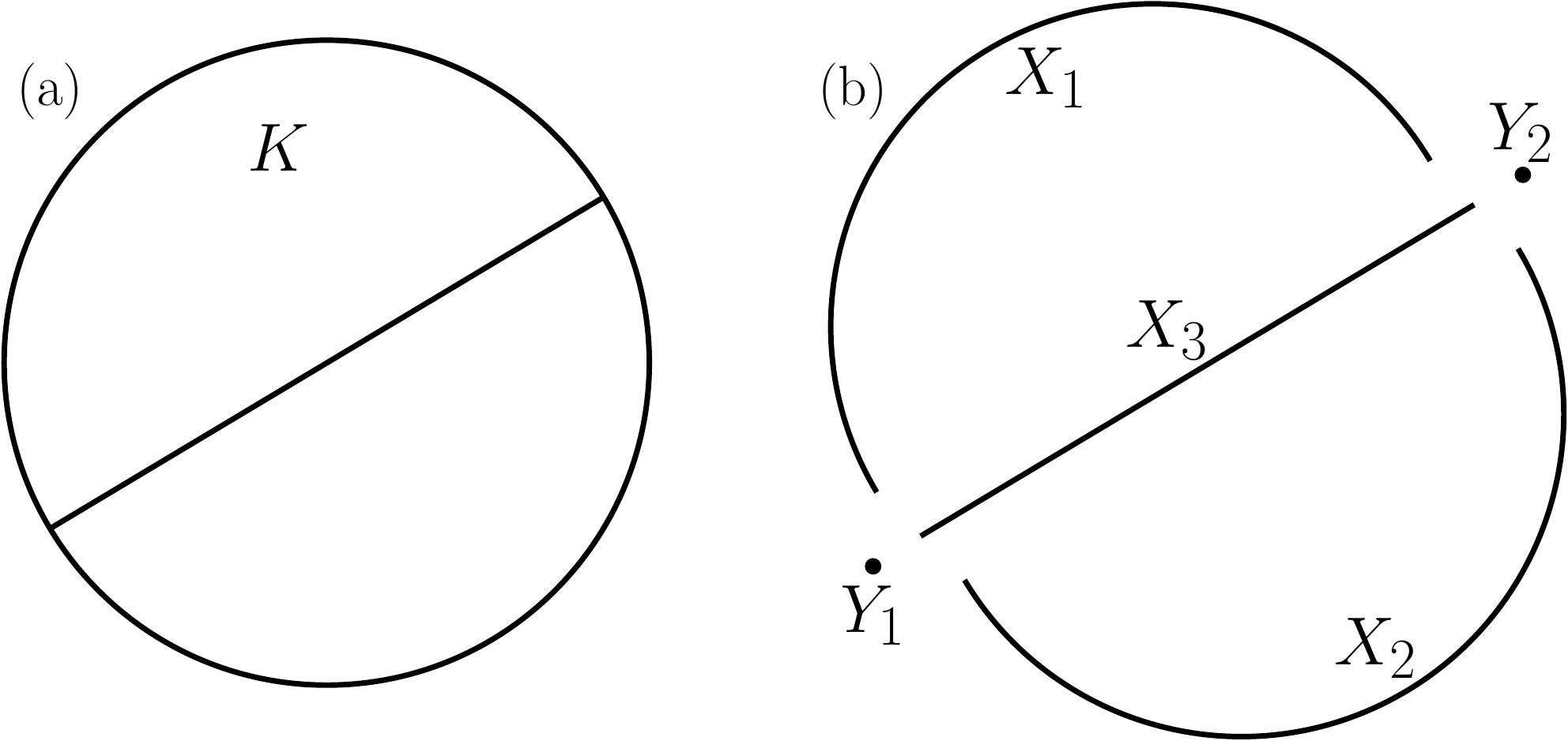}

    \vspace{3ex}
    \includegraphics[width=0.9\textwidth]{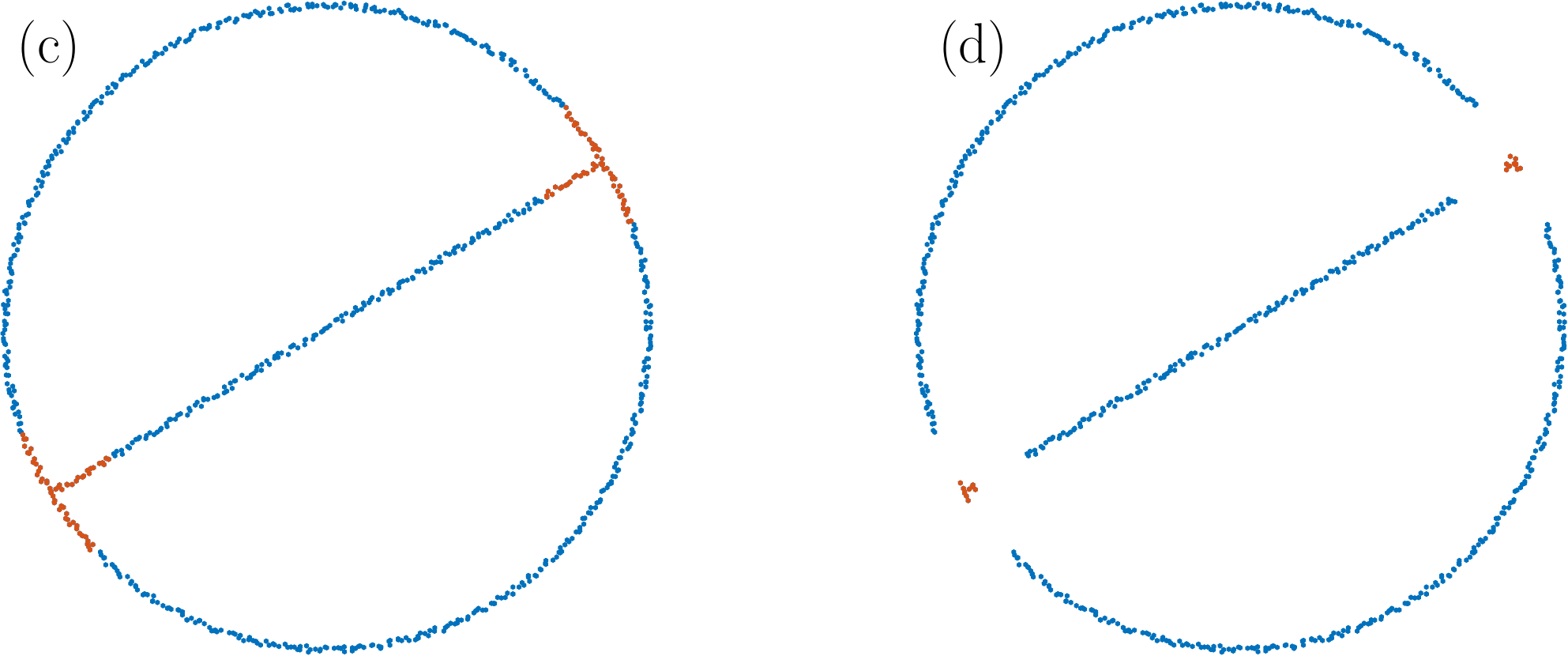}
    \caption{\label{fig:num_example}
    An example of recovering local homology from a sample of a
    stratified set using Vietoris-Rips complexes. (a) A simple
    stratified set, $K$, consisting of a unit circle and a chord passing
    through its center. (b) Strata of $K$, consisting of two open arcs,
    an open line segment, and two points.  (c) An $\e$-sample of $K$,
    with $\e<0.019$. Blue and red colors indicate points where
    $1$-dimensional local homology was estimated to have rank $1$ and
    $2$, respectively. (d) The part of the sample where local homology
    was estimated correctly.  To simplify the computation, coefficients
    in $\Z_2$ were used. The values for the scales of Vietoris-Rips
    complexes, as denoted in Corollary \ref{cor:intro}, were chosen as
    follows: $\delta(\e)=0.06$, $r(\e)=0.116$, $R(\e)=0.175$.
    }
\end{figure}

The rest of the paper is structured as follows. Section
\ref{sec:prelim} provides background information along with some
useful results, and introduces the
classes of stratified sets that we shall be considering. In Section
\ref{sec:seemly}, we define the notions of a weak, moderate, and strong local homological
seemliness. We also show that relative homology computed at points of an $\e$-sample of
a locally homologically seemly stratified set using Vietoris-Rips or \v{C}ech
complexes at certain scales recovers true local homology of the set at
all points of the sample except for a small fraction.
Section \ref{sec:seemly_strat} is dedicated to proving that all the
classes of stratified sets that we are considering, which include
Whitney stratified sets, are locally homologically seemly, with stronger
versions of seemliness for nicer sets. Section
\ref{sec:conclusion} concludes the paper.
To improve the flow of the exposition, we moved the proofs of some
technical lemmas to the Appendix.

\section{Preliminaries}
\label{sec:prelim}
Before delving into the details of our exposition, we need to introduce
some nomenclature and state a few useful results that we shall be
relying upon later. We assume familiarity with basic algebraic topology, in
particular homology theory, and refer the reader to such texts as
\cite{bredon2013,dold1995,eilenberg.steenrod1952},
if a refresher on the subject is needed. A more complete
background information on other relevant topics, e.g. the theory of
stratified spaces and metric geometry, can be found in such
comprehensive texts as \cite{burago.etal2001,pflaum2001}.

\subsection{General notation and problem description}
Throughout the paper, $(\X, d)$ will denote a metric space, which we
shall always assume to be complete and locally compact. We shall also
assume that the metric $d$ is such that for any $x,y\in\X$ there is a
midpoint, i.e. a point $z\in\X$ such that $d(x,y)=2d(x,z)=2d(y,z)$.
Given $K\subseteq\X$, we denote by $\cl{K}$ the closure of $K$, by
$\intr{K}$ the interior of $K$, and by $K^c=\X-K$ the complement of $K$
in $\X$. Here and throughout the paper, the minus sign employed as a
binary operation on sets denotes the usual set difference.
It will be convenient to denote $\R^n_+ =
\{(x_1,\ldots,x_n)\in\R^n\,|\, x_i\geq 0, i=1,\ldots,n\}$. Given
$\e\in\R_+$, we let
\begin{align*}
    B_{\e}(K) &= \{x\in\X \,|\, d(x,K)<\e\},\\
    D_{\e}(K) &= \{x\in\X \,|\, d(x,K)\leq\e\},\\
    S_{\e}(K) &= \{x\in\X \,|\, d(x,K)=\e\},
\end{align*}
where $d(x,K)=\inf_{y\in K}{d(x,y)}$. Thus, for $x\in\X$,
$B_{\e}(x)$, $D_{\e}(x)$, and $S_{\e}(x)$ denote the open ball, the
closed ball, and the sphere of radius $\e$ centered at $x$,
respectively. Throughout the paper, we shall adopt the convention
$\inf{\emptyset}=\infty$ and $\sup{\emptyset}=-\infty$, which
immediately implies that $D_{\e}(\emptyset) =
B_{\e}(\emptyset)=\emptyset$ for $\e\in\R_+$.

The aforementioned assumptions on the metric space $(\X,d)$ guarantee
existence of shortest paths between any two points \citep[Theorem
2.4.16]{burago.etal2001}. We shall call a set $A\subseteq \X$
\emph{strongly convex} if any two points $x,y \in A$ are connected by a
unique shortest path lying completely in $A$ and depending continuously
on the end points. The space $\X$ will be called \emph{locally strongly
convex} if all sufficiently small open (and hence closed) balls at any
point are strongly convex. In particular, if $\X$ has curvature bounded
from above then it is locally strongly convex \citep[Proposition
9.1.17]{burago.etal2001}. The convexity radius at $x\in\X$, denoted
$\conv(x)$, is the supremum over all $\delta$ such that $B_{\delta}(x)$
is convex. The convexity radius of $K\subseteq\X$ is
$\conv(K)=\inf\{\conv(x)\,|\, x\in K\}$. In a locally strongly convex
space, $\conv(K)>0$ if $K$ is compact. For the rest of the paper, $\X$
will be assumed locally strongly convex.

If $\X$ is a Riemannian manifold and $M\subseteq\X$ is a submanifold,
$T_xM$ will denote the tangent space to $M$ at $x\in M$. We will make
use of the notion of transversality. Recall that two submanifolds
$M,N\subseteq\X$ intersect transversally, denoted $M\pitchfork N$,
if for any $x\in M\cap N$ we have $T_xM+T_xN = T_x\X$.

Let $K\subseteq\X$ be a compact set. We shall use the phrase
\emph{local homology of $K$ at $x\in K$} to refer to the local
singular homology groups of all homological dimensions. More
specifically, we define
$\h(K, K-\{x\}) = \bigoplus_{\ell}{H_{\ell}(K, K-\{x\})}$, with maps
between such objects being direct sums of maps between the homology
groups of each dimension. By \emph{local homology of $K$} we shall mean
the collection $\{\h(K, K-\{x\}) | x\in K\}$.
Throughout the paper, homology groups are assumed to have coefficients
in $\Z$.

Our goal is to estimate (in a
very broad sense) the local homology of $K$ when only
an $\e$-sample of $K$ is available. Given $\e>0$, an $\e$-sample of $K$
is a finite set $P\subset\X$ such that $d_{H}(K,P)<\e$, where $d_{H}$
denotes the Hausdorff distance:
$$
d_{H}(K,P) = \max \left\{ \inf\{
    \delta>0 \,|\, K\subseteq B_{\delta}(P) \}\right.,
    \left.\inf \{
    \delta>0 \,|\, P\subseteq B_{\delta}(K) \} \right\}.
$$
Such a $P$ is
often called a noisy $\e$-sample, since we do not require $P\subseteq K$.
If the latter condition is satisfied, we say that $P$ is noise-free.
Note that for $\delta\geq 0$ we have
$$
D_{\delta+\e}(P)\supseteq D_{\delta}(K)\quad
\text{ and }\quad
D_{\delta+\e}(P) - B_r(p)\supseteq D_{\delta}(K)-B_R(x),
$$
where $p\in P$, $x=\argmin_{z\in K}{d(p,z)}$, and $R\geq
r+\e$. Thus, it is reasonable to expect that for a ``sufficiently nice'' set $K$,
the relative homology groups $\h(D_{a}(P), D_{a}(P)-B_{b}(p))$ (with
appropriately chosen $b>a>0$)
capture the local homology of $K$ at $x$. In what follows, it
will be convenient to refer to the size, $\delta$, of a neighborhood
$D_{\delta}(K)$ as a \emph{global scale}, and to the radius of the ball
$B_{R}(x)$ as a \emph{local scale}.

\subsection{A simple algebraic consideration}
A majority of homology inference results in topological data analysis rely on the following
simple observation: if a sequence of group isomorphisms,
$$A_1\to A_2\to A_3,$$
factors through groups $B_1$ and $B_2$ to form a sequence
$$A_1\to B_1\to A_2\to B_2\to A_3,$$
then the image of the resulting homomorphism
from $B_1$ to $B_2$ is isomorphic to $A_i$, $\im{(B_1\to B_2)}\approx
A_i$, $i=1,2,3$. For example, if one can construct three nested neighborhoods of the set of
interest $K$ which capture its true homology and interleave with two
nested neighborhoods of the sample $P$, then the above result tells us
that the correct homology can be recovered by looking at the image of
the inclusion induced homomorphisms between the homology groups of the
two neighborhoods of $P$.

We shall need a slight generalization of this observation.
\begin{lemma}
    \label{lem:homob}
    Suppose that a sequence of group homomorphisms
    $$
    A_0\overset{\vphi_0}{\to} A_1\overset{\vphi_1}{\to}
    A_2\overset{\vphi_2}{\to} A_3
    $$
    factors through groups $B_1$ and $B_2$ to form a sequence
    $$
    A_0\overset{\psi_0}{\to} B_1\overset{\psi_1}{\to}
    A_1\overset{\vphi_1}{\to}
    A_2\overset{\gamma}{\to} B_2\overset{\psi_2}{\to} A_3,
    $$
    and is such that restrictions $\im{\vphi_{i-1}}\to\im{\vphi_i}$,
    $i=1,2$, are isomorphisms. Then $\im{(B_1\to B_2)}$ $\approx$
    $\im{\vphi_i}$, $i=0,1,2$.
\end{lemma}
\begin{proof}
    Clearly, $\im{\vphi_0}=\im{\psi_1\circ\psi_0} \subseteq
    \im{\psi_1}$ and $\im{\vphi_1\circ\psi_1} \subseteq
    \im{\vphi_1}$, and since the restriction
    $\im{\vphi_{0}}\to\im{\vphi_1}$ is an isomorphisms, we get
    $\im{\vphi_1\circ\psi_1} = \im{\vphi_1}$. Also, the restriction
    $\im{\vphi_1}\to\im{\gamma}\circ\vphi_1$ is an isomorphisms because
    $\im{\vphi_1}\to\im{\vphi_2}$ is. Hence,
    $\im{(B_1\to B_2)}=\im{\gamma\circ\vphi_1\circ\psi_1}\approx\im{\vphi_1}$.
\end{proof}

This result allows us to relax the requirement that (pairs of) nested
neighborhoods of the set of interest $K$ (and a point $x\in K$) capture
the true (local) homology -- it is sufficient that the images of the
inclusion induced homomorphisms capture it. More specifically, we need
to look for nested neighborhoods $D_{\delta_i}(K)$,
$B_{r_i}(x)$, $i=0,1,2,3$, with $\delta_{i}\leq\delta_{j}$ and $r_{i}\geq
r_{j}$ for $i\leq j$, such that the images of homomorphisms
$$
\h(D_{\delta_i}(K), D_{\delta_i}(K) - B_{r_i}(x)) \to
\h(D_{\delta_{j}}(K), D_{\delta_{j}}(K) - B_{r_{j}}(x)),\quad i<j,
$$
are isomorphic to $\h(K, K-\{x\})$. Then we can try to interleave these
neighborhoods with the corresponding neighborhoods of the $\e$-sample $P$.

\subsection{\v{C}ech and Vietoris-Rips complexes}
When performing actual computations, one employs combinatorial structures
which capture topology of the neighborhoods of interest. Two of such
structures are the \v{C}ech and Vietoris-Rips simplicial complexes.

If $\X$ is a metric space, $P\subseteq\X$, and $\alpha>0$, the \v{C}ech
complex over $P$ at a scale $\alpha$, $\cech_{\alpha}(P)$, is an
\emph{abstract simplicial complex} consisting of (abstract) simplices, i.e. finite
subsets, $\sigma\subseteq P$
such that $\cap_{x\in\sigma}{D_{\alpha}(x)} \neq \emptyset$. In other
words, it is the \emph{nerve} of the collection of balls $\{D_{\alpha}(x)\,|\,
x\in P\}$ \citep[see e.g.][]{eilenberg.steenrod1952}.
The Vietoris-Rips complex, $\rips_{\alpha}(P)$, consists of
all simplicies $\sigma\subseteq P$ whose edges belong to
$\cech_{\alpha}(P)$.

If $P$ is finite and the balls $D_{\alpha}(x)$, $x\in P$, are strongly
convex, then $D_{\alpha}(P)$ and $\cech_{\alpha}(P)$ are homotopy
equivalent \citep[see e.g.][]{dugundji1967}. Vietoris-Rips complexes are
generally not homotopy equivalent to the corresponding \v{C}ech
complexes, but they  are easier to construct and satisfy the following
interleaving condition: $\cech_{\alpha}(P)$ $\subseteq$
$\rips_{\alpha}(P)$ $\subseteq$ $\cech_{s\alpha}(P)$, where $s=2$ in
general, and $s=\sqrt{\frac{2n}{n+1}}$ if $\X$ is an $n$-dimensional
Euclidean space \citep[see e.g.][Theorem 2.5]{de2007coverage}. If the
dimension of the Euclidean space is not specified, one can safely take
$s=\sqrt{2}$.

We shall use \v{C}ech and Vietoris-Rips complexes in the setting where
$P$ is an $\e$-sample of a compact set $K\subseteq\X$.  The goal is to
recover the local homology at points of $K$ using the relative homology
of appropriately chosen subcomplexes of the \v{C}ech and Vietoris-Rips
complexes over $P$.  The required scales for the simplicial complexes
may differ depending on the type of the $\e$-sample (noisy or
noise-free) and the space $\X$ (Euclidean space or not). Hence, it will
be convenient to introduce two constants (depending on $P$ and $\X$,
respectively) that capture these differences. Throughout the paper, we
let
\begin{equation}\label{nom:const}
    t=\left\{
        \begin{aligned}
            0,&\quad\text{ if } P \text{ is noise free},\\
            1,&\quad\text{ otherwise},
        \end{aligned}
      \right.
      \quad\quad
    s=\left\{
        \begin{aligned}
            \sqrt{2},&\quad\text{ if } \X \text{ is a Euclidean space},\\
            2,&\quad\text{ otherwise}.
        \end{aligned}
      \right.
\end{equation}
It will also be convenient
to introduce the some helpful notation for local and relative homology
groups. Throughout the paper, we let 
\begin{equation}\label{nom:hom}
\begin{aligned}
    F(a,b) &= \h(D_{a}(K), D_{a}(K)-B_{b}(x)),\quad
    \text{ for } a>0 \text{ or } b>0,\\
    F(0,0)&=\h(K, K-\{x\}),\\
    C(a,b) &= \h(\cech_{a}(P), \cech_{a}(P-B_{b}(p))),\\
    V(a,b) &= \h(\rips_{a}(P), \rips_{a}(P-B_{b}(p))).
\end{aligned}
\end{equation}
The dependency of the left hand sides
on points $p\in P$ and $x\in K$ has been suppressed, since it will be
either stated or clear from the context how they should be chosen.

We now can state a lemma that lays a foundation for the subsequent use
of \v{C}ech and Vietoris-Rips complexes.
\begin{lemma}\label{lem:interleave}
    Suppose that $K\subseteq\X$ is compact and $P\subset\X$ is an
    $\e$-sample of $K$. Let $p\in P$, and let
    $x\in K$ be a point closest to $p$.
    For $\delta\geq 0$, take
    \begin{equation*}
        R>(1+s)\delta+(1+s+2t)\e,
    \end{equation*}
    and assume
    $s(\delta+\e)<\conv(P)$.
    Then inclusion induced homomorphisms
    \begin{equation*}
        F(\delta,R) \to F(\beta_c,r_c)
        \quad\text{and}\quad
        F(\delta,R) \to F(\beta_v,r_v),
    \end{equation*}
    where
    \begin{align*}
        &\beta_c\geq \delta+(1+t)\e, \quad r_c\leq R-2\delta-(2+2t)\e,\\
        &\beta_v\geq s\delta+(s+t)\e,\quad r_v\leq R-(1+s)\delta-(1+s+2t)\e,
    \end{align*}
    factor through
    \begin{equation*}
        C(\delta+\e, r+\delta+(1+t)\e)
        \quad\text{and}\quad
        V(\delta+\e, r+s\delta+(s+t)\e),
    \end{equation*}
    respectively:
    \begin{equation}\label{eq:c_interleave_n}
    \begin{tikzcd}[column sep=small, row sep=tiny]
        F(\delta,R) \ar[r] &C(\delta+\e, r_c+\delta+(1+t)\e) \ar[r] &
        F(\beta_c, r_c)
    \end{tikzcd}
    \end{equation}
    \begin{equation}\label{eq:r_interleave_n}
    \begin{tikzcd}[column sep=small, row sep=tiny]
        F(\delta,R) \ar[r] &V(\delta+\e, r_v+s\delta+(s+t)\e) \ar[r] &
        F(\beta_v, r_v)
    \end{tikzcd}
    \end{equation}
\end{lemma}
\begin{proof} 
    Note that if $P\subseteq K$ then $p=x$, so we have
    $B_{b}(x)$ $\supseteq$ $B_{b-t\e}(p)$.
    The triangle inequality also yields
    \begin{gather*}
        D_{a}(K)\subseteq D_{a+\e}(P) \subseteq D_{a+(1+t)\e}(K),\\
        D_{a}(P)-B_{b}(p) \subseteq D_{a}(P-B_{b-a}(p)) \subseteq
        D_{a}(P)-B_{b-2a}(p).
    \end{gather*}
    Since
    $\cech_{\alpha}(Q)$ is homotopy equivalent to $D_{\alpha}(Q)$ for
    any $Q\subseteq P$,
    $\alpha<\conv(P)$, the five-lemma \citep[see e.g.][Lemma 5.10]{bredon2013} yields
    $\h(\cech_{\alpha}(P), \cech_{\alpha}(Q))$
    $\approx$
    $\h(D_{\alpha}(P), D_{\alpha}(Q))$.
    Denoting
    \begin{align*}
        G(a,b) &= \h(D_{a}(P), D_{a}(P)-B_{b}(p)),\\
        J(a,b) &= \h(D_{a}(P), D_{a}(P-B_{b}(p))),
    \end{align*}
    and noting that $r_c+\delta+(1+t)\e$ $\leq$ $R-\delta-(1+t)\e$
    we obtain
    $$
    \begin{tikzcd}[column sep=small, row sep=tiny]
        &[-4em]C(\delta+\e, r_c+\delta+(1+t)\e)\ar[d, phantom, sloped,"\approx"]&[-4em]\\
        &[-4em]J(\delta+\e, r_c+\delta+(1+t)\e) \ar[dr] &[-4em]\\
        J(\delta+\e, R-\delta-(1+t)\e) \ar[ur] &[-4em]&[-4em]
        G(\delta+\e, r_c+t\e) \ar[d]\\
        G(\delta+\e, R-t\e) \ar[u] &[-4em]&[-4em]
        F(\delta+(1+t)\e, r_c)\ar[d]\\
        F(\delta, R) \ar[u] &[-4em]&[-4em]
        F(\beta_c, r_c)
    \end{tikzcd}
    $$
    which implies \eqref{eq:c_interleave_n}. Note that it is
    enough to have $\delta+\e<\conv(K)$ for this result to hold. 

    Taking into account that $\cech_{\alpha}(Q)$ $\subseteq$
    $\rips_{\alpha}(Q)$ $\subseteq$ $\cech_{s\alpha}(Q)$ for any
    $Q\subseteq
    P$, we obtain
    $$
    \begin{tikzcd}[row sep=tiny]
        C(\delta+\e, r_v+s\delta+(s+t)\e)\ar[r]
        \ar[d,phantom, sloped, "\approx"]&
        V(\delta+\e, r_v+s\delta+(s+t)\e)\ar[d]\\
        J(\delta+\e, r_v+s\delta+(s+t)\e)&
        C(s\delta+s\e, r_v+s\delta+(s+t)\e)
        \ar[d,phantom, sloped, "\approx"]\\
        J(\delta+\e, R-\delta-(1+t)\e) \ar[u]&
        J(s\delta+s\e, r_v+s\delta+(s+t)\e) \ar[d]\\
        G(\delta+\e, R-t\e) \ar[u] &
        G(s\delta+s\e, r_v+t\e)\ar[d]\\
        &F(s\delta+(s+t)\e, r_v)\ar[d]\\
        F(\delta, R) \ar[uu] &
        F(\beta_v, r_v)
    \end{tikzcd}
    $$
    which implies \eqref{eq:r_interleave_n}.
\end{proof}
It is worth mentioning that, in the above lemma, the noisy case does not
require $x$ to be the closest point -- we simply need $d(x,p)<\e$. Also,
we can notice that the amounts by which the scales of the complexes and
the radii of the balls have to change follow a certain pattern. More
specifically, the quantities that stand out are $c\delta+(c+t)\e$ and
$(1+c)\delta+(1+c+2t)\e$, where $c$ is equal to either $1$ or $s$.
Since we may need to employ these kind of quantities multiple times,
we introduce the notation
\begin{equation}\label{nom:func}
\begin{aligned}
    g_c(a,b) &= ca+(c+t)b,\\
    f_c(a,b) &= g_c(a,b)+g_1(a,b) =\\
    &=(1+c)a+(1+c+2t)b,
\quad\quad\quad\quad
    c\in\{1,s\},
\end{aligned}
\end{equation}
which will be used throughout the paper.

\subsection{Stratified sets}
To ensure a feasibility of the above approach to local homology
recovery, at least for small enough $\e$, we need to impose some
restrictions on the set $K$, and on the way $K$ is
embedded in $\X$. To address the latter, we shall assume that $K$ is a
neighborhood retract, that is, there is a neighborhood $U\supseteq K$ and
a continuous map $\pi:U\to K$ such that $\pi(x)=x$ for all $x\in K$.
As for $K$ itself, we shall require that
it admit a \emph{tame homology stratification}. By a stratification of $K$ we
mean a locally finite collection $\st$ of pairwise disjoint, locally
closed subsets of $K$ such that $K=\cup_{X\in\st}{X}$ and the
\emph{Frontier Condition} is satisfied:
\begin{equation*}
\text{for all }
    X,Y\in\st, X\cap\cl{Y}\neq\emptyset \implies X\subseteq\cl{Y}.
\end{equation*}
A set $X\in\st$ is
called a \emph{stratum}. A stratification $\st$ is a tame homology
stratification if it satisfies the following conditions:
\begin{enumerate}
    \item Each $X\in\st$ is a finite dimensional homology manifold. That
        is, for any $x\in X$, $\hl(X, X-\{x\})$ is trivial for all
        $\ell\geq 0$ except one -- the dimension of $X$ -- in which case
        it is isomorphic to $\Z$.
    \item For any $x\in X\in\st$, the inclusion induced homomorphisms
        \begin{gather*}
            \hl(K, K-B_{r}(x)) \to \hl(K, K-\{x\}), \text{ and}\\
            \hl(K, K-B_{r}(x)) \to \hl(K, K-B_{\rho}(y)) \to \hl(K, K-\{y\})
        \end{gather*}
        are all isomorphisms as long as $r>0$ is sufficiently small,
        $y\in B_{r}(x)\cap X$, and $0<\rho<r-d(x,y)$. Moreover, each
        $\hl(K,K-\{x\})$, $\ell\geq 0$, is finitely generated.
\end{enumerate}
Note that these conditions imply that local homology
groups remain constant along the strata. We shall consider only tame
homology stratifications, and will omit the qualifier ``tame'' for
convenience.

Assume now that a set $K\subseteq\X$ is endowed with a stratification
$\st$. The Frontier Condition induces a partial order on the
stratification: given $X,Y\in\st$, we define $X\leq Y$ if $X\subseteq
\cl{Y}$.  It follows that for each $Y\in\st$ we have $\cl{Y} =
\cup_{X\leq Y}{X}$. We say that a stratum $X\in\st$ has \emph{height}
$k$, and denote it by $\height{X}=k$, if $k$ is the largest integer such
that there exist $X_0,\ldots,X_k\in\st$ satisfying $X_0\leq\cdots\leq
X_k=X$. In other words, the height of $X\in\st$ is one less than the
size of the longest chain in $\st$ having $X$ as the maximal element.
As an example, consider the stratification shown in Figure
\ref{fig:num_example}(b). It has two strata of height $0$
and three strata of height $1$.
We shall denote by $\st_k$ the collection of all the strata of height
$k$, that is, $\st_k = \{X\in\st\,|\,\height{X}=k\}$.  Clearly, every
minimal element of $\st$ has height zero, so $\st_0$ contains all the
minimal elements.

Suppose we have another metric space $\Y$ and a set $L\subseteq\Y$ with
a stratification $\mathcal{R}$. We say that a map $f:\X\to\Y$ is
\emph{stratum preserving} (or stratified) if $f(K^c)\subseteq L^c$, and
for any $X\in\st$ there is $Y\in\mathcal{R}$ such that $f(X)\subseteq Y$.
Similarly, given a (not necessarily stratified) set $A\subseteq\Y$ and a map
$F:A\times[0,1]\to\X$, we say that $F$ is stratum preserving if for any
$x\in A$ either $F(\{x\}\times[0,1])\subseteq K^c$ or there is $X\in\st$
such that $F(\{x\}\times[0,1])\subseteq X$. We say that $F$ is
\emph{nearly stratum preserving} if it is stratum preserving on
$A\times[0,1)$.

Existence of a homology stratification of $K$ implies only very
mild restrictions on the geometric behavior of the neighborhoods of
strata in $K$. One could make those restrictions a little stronger by
requiring strata to be (topological) manifolds satisfying certain
homotopy based or homeomorphism based compatibility conditions, as do
the homotopically stratified spaces of Quinn \citep{quinn1988} or the locally
cone-like TOP stratified sets of Siebennman \citep{siebenmann1972}. However, it will be
more instructive to investigate how local homology recovery improves
when significantly stronger geometric restrictions on the strata are
imposed.

Perhaps the most well known class of ``nice'' stratified sets are
Whitney stratified sets. In this case, we assume $\X$ to be a smooth,
complete Riemannian manifold. A stratification $\st$ of a set
$K\subseteq\X$ is called a \emph{Whitney stratification} if each stratum is a smooth
manifold, and for any two strata $Y\leq X$ the following conditions
hold: whenever $x_i\in X$ and $y_i\in Y$ are two sequences converging to
$y\in Y$ such that the tangent spaces $T_{x_i}X$ converge (in the
corresponding Grasmannian) to a space $\tau$, and (with respect to some
local coordinate system on $\X$) the secant lines  $l_i=\overline{x_iy_i}$
converge (in the corresponding projective space) to a line $l$, we have
\begin{itemize}
    \item[(a)] $T_yY\subseteq \tau$
    \item[(b)] $l\subseteq \tau$.
\end{itemize}

Stratified sets satisfying the above condition (a) or (b) are called
(a)-regular or (b)-regular, respectively. It is well know that
(b)$\implies$(a) \citep[see e.g.][]{mather2012notes}. If $K\subset\X$ is a
Whitney stratified set, it may be useful to employ a uniform
view and regard the manifold $\X$ as a Whitney stratified set whose
strata are $\X-K$ and the strata of $K$.

Besides also being homology
stratified, Whitney stratified sets have a lot of important properties
\citep{goresky1988strat,pflaum2001}. Pertinent to our problem is the fact that if $\st$ is a Whitney
stratification of a set $K\subseteq\X$ and $L\subseteq K$ is a closed union
of strata, then there exists a neighborhood $U\supseteq L$ in $\X$ and a
nearly stratum preserving continuous map $F:U\times[0,1]\to U$ providing a strong
deformation retraction of $U$ onto $L$, i.e., $F(x,0)=x$ and $F(x,1)\in
L$ for $x\in U$, $F(y,t)=y$ for $y\in L$, $t\in[0,1]$
(see e.g. \cite{pflaum2017} as well as Theorem 3.9.4 in \cite{pflaum2001}).
Another useful fact concerns transverse intersections and
unions of Whitney stratified sets. If $\st$ and $\mathcal{R}$ are Whitney
stratifications of subsets $K,L\subseteq\X$, respectively, we say that $K$
intersects $L$ transversely if $X$ intersects $Y$ transversally for any
$X\in\st$, $Y\in\mathcal{R}$. In such a case, the stratification
$$\st\cap_t \mathcal{R} = \cup{\{X\cap Y\}},$$
where the union is taken
over all $X\in\st$ and $Y\in\mathcal{R}$ with non-empty intersection, is
a Whitney stratification of $K\cap L$ \citep[see e.g.][]{cheniot1972sections}.
In other words, a transverse intersection of two Whitney stratified sets
is a Whitney stratified set whose strata are
intersections of the strata of the two sets.
Let us define
$$
\st\cup_t\mathcal{R} = \bigcup\big(\{X\cap
Y\}\cup\{X-Y\}\cup\{Y-X\}\big),
$$
where the union is taken over all $X\in\st$ and $Y\in\mathcal{R}$ with
non-empty intersection. The proof of the following lemma can be found in
the Appendix.
\begin{lemma}\label{lem:trans_union}
    Let $K, L\subseteq\X$  be two Whitney stratified sets with
    stratifications $\st$ and $\mathcal{R}$, respectively. Suppose that
    $K$ and $L$ intersect transversely. Then $K\cup L$ is a Whitney
    stratified set with stratification $\st\cup_t\mathcal{R}$.
\end{lemma}
Thus, the union of two transversely intersecting Whitney stratified sets is a
Whitney stratified set whose strata are the intersections and the
differences of the strata of the two sets.

\subsection{Distance function and weak feature size}
An alternative way to impose some geometric regularity on a compact set
$K\subset\X$ is through the properties of the distance function
$$
d_K:\X\to\R_+,\quad
d_K(x) = \inf_{y\in K}{d(x,y)}.
$$
Assuming that $\X$ is
a smooth Riemannian manifold, a point $x\in\X-K$ is called
\emph{regular} for $d_K$ if there is a unit vector $v\in T_x\X$ such
that for any shortest unit speed geodesic $\gamma$ connecting $x$ to the
set of closest points in $K$, we have $\angle(v,\dot\gamma(0)) >
\frac{\pi}{2}$ \citep[see e.g.][]{grove1993,petersen2006}. Otherwise, $x$ is a \emph{critical}
point, and the value $r\in\R_+$ is called critical if $d_K^{-1}(r)$
contains a critical point. The \emph{weak feature size} of $K$,
$\wfs(K)$, is the infimum of all positive critical values of $d_K$. If
$\wfs(K)>0$, that is, there exists $\e>0$ such that $d_K^{-1}((0,\e))$
contains only regular points, then
$D_{\delta_1}(K)=d_K^{-1}([0,\delta_1])$ is isotopic to
$D_{\delta_2}(K)=d_K^{-1}([0,\delta_2])$ for any
$\delta_1,\delta_2\in(0,\e)$ \citep{grove1993}. Note that a Whitney
stratified set may not have a positive weak feature size.  Also, a
compact set with a positive weak feature size may not admit a Whitney
stratification. Thus, it may be reasonable to combine the two conditions
and consider Whitney stratified sets which have a positive weak feature
size.

\section{Homological seemliness}
\label{sec:seemly}

Let us now fix a compact homology stratified neighborhood retract
$K\subseteq\X$ with a stratification $\st$. First, we try to determine
conditions that allow us to find nested neighborhoods of $K$ and a fixed
$x\in K$ which can be used with Lemmas \ref{lem:homob} and
\ref{lem:interleave}. We later show show that analogous conditions can
be obtained for all of $K$, and that these conditions do hold under the
assumptions that we made regarding $K$ and $\X$.

In addition to the assumptions on $K$ and $\X$,
this section will rely on notation \eqref{nom:const},
\eqref{nom:hom}, and \eqref{nom:func}, introduced in Section
\ref{sec:prelim}. Also, $P$ will always denote an $\e$-sample of $K$,
$p\in P$ will be an arbitrary point, and, where appropriate, $x\in K$
will be either a point with $d(x,p)<\e$, if $P$ is noisy, or the closest
point to $p$, if $P$ is noise-free. Any additional restrictions on $x$
will be states explicitly.

\subsection{Homological seemliness at a point}
Given $\e\in\R_+$, let $A(\e)=\{(R,r,\alpha) \in\R^3_+ | R\geq r >
\alpha\geq \e\}\subset\R^3_+$. For $x\in K$ we define
$A_{x}(\e)$ as the set of all triples $(R,r,\alpha)\in A(\e)$ which
satisfy the following conditions:
\begin{enumerate}
    \item For all $\rho\in(0,R]$, $D_{\rho}(x)$ are topological balls
        (and $S_{\rho}(x)$ are transverse to all the strata, if $K$ is a
        Whitney stratified set).
    \item Inclusions yield the commutative diagram
    \begin{equation}\label{eq:def_A}
        \begin{tikzcd}[row sep=small]
            F(\e,R) \ar[r]\ar[rr,bend left=10, "\jj"]&
            \im{\jj}\ar[r] &
            F(\alpha,r)\\
            F(0,R) \ar[ru, sloped, "\jj\circ\ii" above, sloped, "\approx" below]\ar[u,"\ii"]\ar[r,"\approx"]&
            F(0,\rho) \ar[r,"\approx"]&
            F(0,0).
    \end{tikzcd}
    \end{equation}
\end{enumerate}
Note that the above diagram imposes constraints on the triples
$(R,r,\alpha)$ by requiring that certain inclusion induced homomorphisms
be isomorphisms.  The presence of the diagonal arrow in diagram
\eqref{eq:def_A} is equivalent to saying that $\ii$ is injective and the
restriction $\im{\ii}\to\im{\jj}$ is an isomorphism. In the case $\e=0$,
the imposed constraints imply that $A_{x}(0)$ contains triples
$(R,r,\alpha)$ such that inclusion induced homomorphisms
$$
\h(K, K-B_{\rho}(x)) \to \h(D_{\alpha}(K),
D_{\alpha}(K)-B_{\rho}(x)),\quad\rho\in[r,R],
$$
are injective (with images isomorphic to the true local homology at
$x$).

Intuitively, one can regard the sets $A_{x}(\e)$ as containing admissible
global and local scales that can \emph{potentially} be used with Lemmas
\ref{lem:homob} and \ref{lem:interleave}. Indeed, suppose $(R,r,\alpha)\in
A_{x}(\e)$. Then any spurious homology classes present in
$\h(D_{\e}(K)$, $D_{\e}(K)$ $-$ $B_{R}(x))$ disappear once we increase the
global scale to $\alpha$ and decrease the local scale to $r$. This
makes the use of Lemma \ref{lem:homob} feasible. To see how to transform the
mere feasibility into an actual result, we need to better understand the
structure of $A_{x}(\e)$.

We start with the simple case $\e=0$. As an example, we explicitly
computed the set $A_{x}(0)$ for the simple stratified set from Figure
\ref{fig:num_example}, with $x$ a $0$-dimensional stratum. It is shown
in Figure \ref{fig:A0}.  More generally, we can prove the following
result.
\begin{figure}
    \centering
    \includegraphics[width=0.6\textwidth]{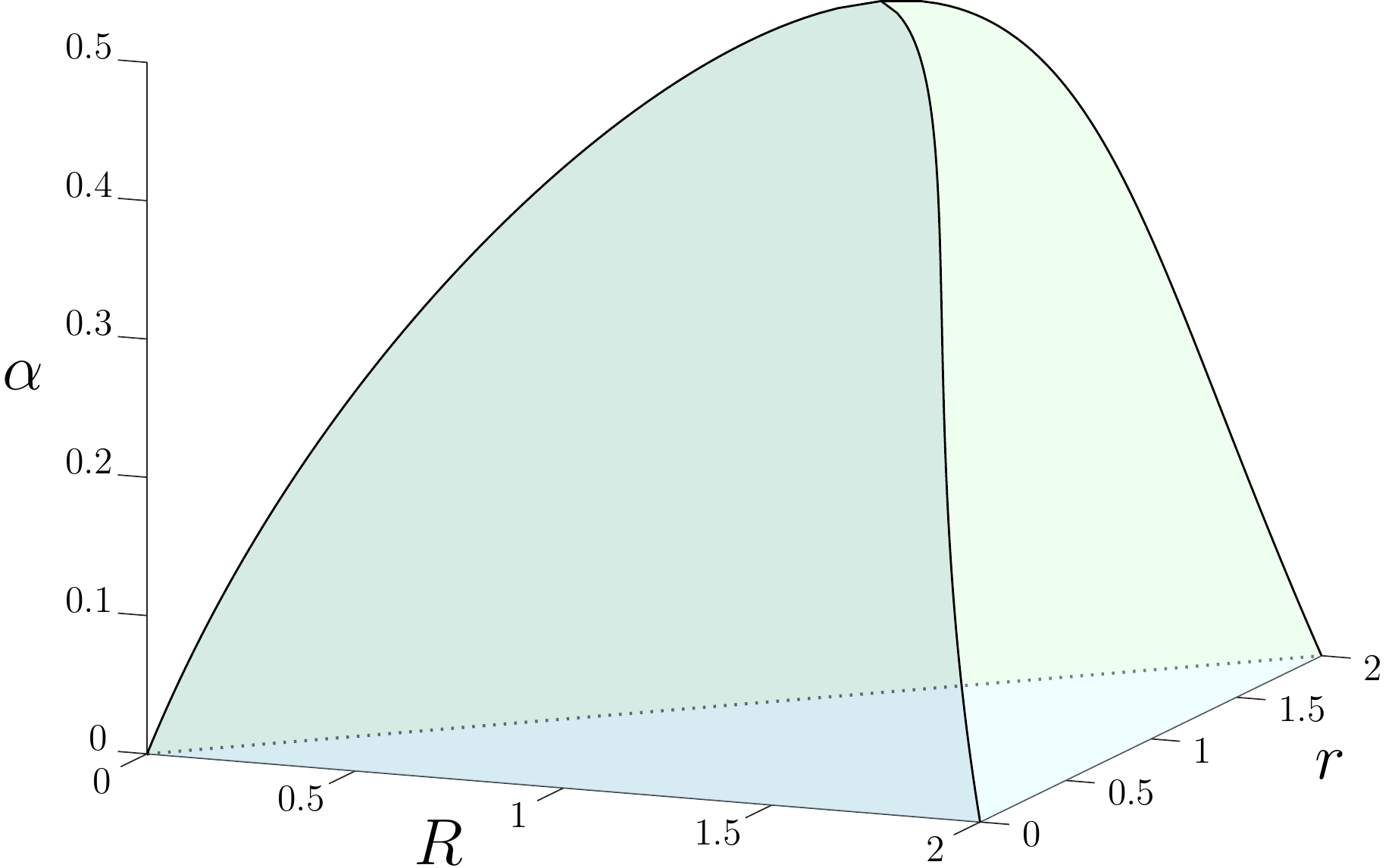}
    \caption{\label{fig:A0}
    The structure of $A_{x}(0)$ for the stratified set from Figure
    \ref{fig:num_example}.
    }
\end{figure}

\begin{lemma}
    \label{lem:A_0}
    For a sufficiently small $\alpha'>0$, the sets of local scales
    $$
        A_{x}(0\,|\,\alpha=\alpha') =
        \{(R,r)\,|\, (R,r,\alpha')\in A_{x}(0)\}
    $$
    is not empty and has the structure of a right isosceles
    triangle (possibly not containing its legs) in the $(R,r)$-plane
    with the legs parallel to the axes, and the hypotenuse lying on the
    diagonal. Moreover, if $\alpha'_1\leq\alpha'_2$ then
    $$
        A_{x}(0\,|\,\alpha=\alpha'_2) \subseteq
        A_{x}(0\,|\,\alpha=\alpha'_1).
    $$
\end{lemma}
The proof of Lemma \ref{lem:A_0} can be found in the Appendix. The lemma
suggests that the structure of $A_{x}(\e)$ may be understood through the
structure of the sets
$$
A_{x}(\e\,|\,\alpha=\alpha') = \{(R,r)\,|\,
(R,r,\alpha')\in A_{x}(\e)\},
$$
which we shall refer to as $\alpha$-sections.


It turns out that the structure of $A_{x}(\e)$ for $\e>0$ is only
slightly more complicated than the structure of $A_{x}(0)$, and can be
nicely described by considering the sections of $A_{x}(\e)$ by lines
parallel to the coordinate axes.  More specifically, we define
\begin{align*}
    A_{x}(\e\,|\, R=R',r=r')&=
    \{\alpha\in\R_+ | (R',r',\alpha)\in A_{x}(\e)\},\\
    A_{x}(\e\,|\, R=R',\alpha=\alpha') &=
    \{r\in\R_+ | (R',r,\alpha')\in A_{x}(\e)\},\\
    A_{x}(\e \,|\, r=r',\alpha=\alpha') &=
    \{R\in\R_+ | (R,r',\alpha')\in A_{x}(\e)\}.
\end{align*}
We shall refer to these sets as line sections of
$A_{x}(\e)$.

\begin{lemma}
    \label{lem:A}\leavevmode
    \begin{enumerate}
        \item \label{lem:A_1} $A_{x}(\e)\subseteq A_{x}(\delta)$
            whenever $0\leq\delta\leq\e$;
        \item \label{lem:A_2} Line sections of $A_{x}(\e)$
            are (possibly degenerate) intervals.
        \item \label{lem:A_3} Suppose $(R_1,r_1,\alpha_1)\in A_{x}(\e)$, and
            let $R_2>R_1$, $r_2<r_1$. Then 
            \begin{align*}
                &A_{x}(\e | R=R_1,\alpha=\alpha_1) \subseteq A_{x}(\e |
                R=R_2,\alpha=\alpha_1),\\
                &A_{x}(\e | r=r_1,\alpha=\alpha_1) \subseteq A_{x}(\e |
                r=r_2,\alpha=\alpha_1)
            \end{align*}
            as long as the corresponding right hand side is not an empty
            set.
        \item \label{lem:A_4} The above properties are preserved under intersections.
    \end{enumerate}
\end{lemma}
The proof of Lemma \ref{lem:A} can be found in the Appendix.  The lemma
implies that $\alpha$-sections $A_{x}(\e | \alpha=\alpha')$ have a
structure reminiscent of a right isosceles triangle. More specifically,
if
\begin{align*}
    R_{u} &= \sup\{R\,|\, \exists r : (R,r)\in A_{x}(\e | \alpha=\alpha')\},\\
    R_{l} &= \inf\{R\,|\, \exists r : (R,r)\in A_{x}(\e | \alpha=\alpha')\},\\ 
    r_{u} &= \sup\{r\,|\, \exists R : (R,r)\in A_{x}(\e | \alpha=\alpha')\},\\
    r_{l} &= \inf\{r\,|\, \exists R : (R,r)\in A_{x}(\e | \alpha=\alpha')\},
\end{align*}
then it is the set bounded by the line segment connecting points $(R_{l},
r_{l})$ and $(R_{u}, r_{l})$, another line segment connecting $(R_{u},
r_{l})$ and $(R_{u}, r_{u})$, and a curve connecting $(R_{l}, r_{l})$
and $(R_{u}, r_{u})$, and such that it is 
a graph (in the $(R,r)$-plane) of a non-decreasing function
$f:[R_{l}, R_{u}]\to[r_l,r_u]$. We should note that this set may not
contain (parts of) its boundary. An example of a possible structure of
an $\alpha$-section is shown in Figure \ref{fig:alpha_sec}.

Looking back at Lemma \ref{lem:interleave}, we see that the ability to
interleave neighborhood of the form
$(D_{\e}(K)$, $D_{\e}(K)$ $-$ $B_{\rho}(x))$ with the appropriate
\v{C}ech or Vietoris-Rips complexes (so that Lemma $\ref{lem:homob}$ can
be invoked) requires a certain range of local scales at appropriately
chosen global scales. More specifically, starting with $(K, K -
B_{R}(x))$ (and an $\e$-sample) we need to increase the global scale and
decrease the local scale by some minimum amounts, so that an appropriate
simplicial complex can be fit in between. This results in a global scale
$\alpha_1$, and a local scale $\rho_1$. These scales need to be changed
further to kill any spurious homology that could have been created, thus
leading to scales $\alpha_2\geq\alpha_1$, $\rho_2\leq\rho_1$. The choice
of the next scale levels, $\alpha_3$ and $\rho_3$, has to be done so
that we, again, can fit our simplicial complex (at a larger scale) as
well as destroy any spurious homology born at the previous scale levels.

The above procedure mandates that, at least for small $\e$, the sets
$A_{x}(\e)$ contain a sufficiently large range of $\alpha$-sections,
each including a sufficiently large range of local scales. To formalize
these ideas, we denote a range of $\alpha$-sets by
$$
A_{x}(\e\,|\,
\alpha\in[\alpha_1,\alpha_2]) = \{(R,r,\alpha)\in A_{x}(\e)\,|\,
\alpha\in[\alpha_1,\alpha_2]\},
$$
call it an $\alpha$-slab, and introduce the following definition.

\begin{definition}
    The collection of sets $A_{x}(\e)$, $\e\geq0$, is called weakly
    seemly if there are $\e_x>0$ and functions
    $\alpha_{x}^{l},\alpha_{x}^{u}:[0,\e_x]\to\R_+$ satisfying the
    following:
    \begin{enumerate}
        \item $\alpha_{x}^{l}$ is non-decreasing and $\alpha_{x}^{u}$ is
            non-increasing;
        \item $\alpha_{x}^{l}(\e)\searrow 0=\alpha_{x}^{l}(0)$ and $\alpha_{x}^{u}(\e)\nearrow
            \alpha_{x}^{u}(0)$ as $\e\to 0$;
        \item $A_{x}(\e\,|\,\alpha=\alpha')\neq\emptyset$ for all
            $\alpha'\in[\alpha_{x}^{l}(\e), \alpha_{x}^{u}(\e)]$;
    \item 
    $\displaystyle 
    \sup_{\alpha'\in[\alpha_{x}^{l}(\e), \alpha_{x}^{u}(\e)]}{
        d_H(A_{x}(0\,|\, \alpha=\alpha'), A_{x}(\e\,|\,
        \alpha=\alpha'))}\to 0\quad\text{as}\quad
    \e\to 0.
    $
    \end{enumerate}
    In this case, we also say that $K$ is weakly homologically seemly at $x$.
\end{definition}

Recalling the structure of $A_{x}(0)$, we see that if $A_{x}(\e)$ is
weakly seemly then for all sufficiently small $\e$ there is a
non-shrinking $\alpha$-slab with $\alpha$-sections that are very close
to right isosceles triangles of Lemma \ref{lem:A_0}. Assuming weak
seemliness, in which case we only consider $\e\in[0,\e_x]$, let us
define
$$
A_{x}([0,\e]\,|\,\alpha\leq\beta) = 
\bigcap_{\e'\in[0,\e]}{
    \bigcap_{\alpha'\in[\alpha_{x}^{l}(\e'),\beta]}{
        A_{x}(\e'\,|\, \alpha=\alpha')}},
$$
where $\beta\in[\alpha_{x}^{l}(\e),\alpha_{x}^{u}(\e)]$. This is the
intersection of all $\alpha$-sections of $A_{x}(\e')$ between
$\alpha_{x}^{l}(\e')$ and $\beta$ for all $\e'\in[0,\e]$. It has the
structure of a typical $\alpha$-section, and if $\e$ is sufficiently
small, it is close to a right isosceles triangle.
To better quantify the difference between the two, we define
\begin{align*}
\RR_x(\e,\beta) &=
\sup\{R\,|\, \exists r : (R,r) \in A_{x}([0,\e] \,|\, \alpha \leq\beta)\},\\
\rr_x(\e,\beta) &=
\inf\{r\,|\, \exists R : (R,r)\in A_{x}([0,\e]\,|\,\alpha\leq\beta)\}.
\end{align*}
In addition, we denote by
$\Delta_x(\e,\beta, \delta)$ the interior of the triangle with vertices
$$
(\rr_x(\e, \beta) + \delta, \rr_x(\e, \beta)), (\RR_x(\e, \beta),
\rr_x(\e, \beta)), (\RR_x(\e, \beta), \RR_x(\e, \beta) - \delta),
$$
and let
$$
\dd_x(\e,\beta) = \inf\{\delta>0 \,|\, \Delta_x(\e, \beta,
\delta) \subseteq A_{x}([0,\e] \,|\, \alpha \leq \beta)\}.
$$
\begin{figure}
    \centering
    \includegraphics[width=0.6\textwidth]{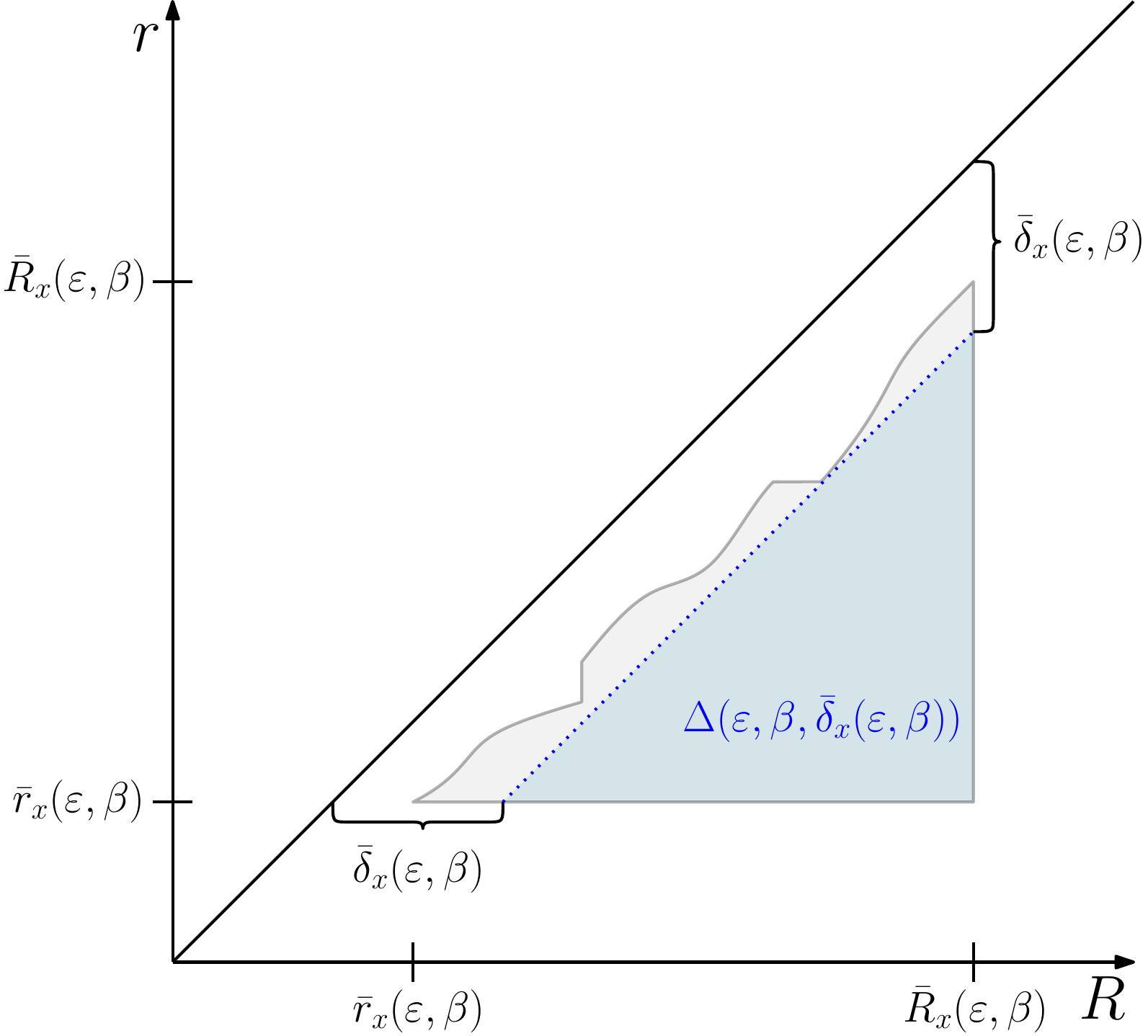}
    \caption{\label{fig:alpha_sec}
    An example of the structure of $A_{x}([0,\e] \,|\, \alpha \leq
    \beta)$, which is also the structure of an $\alpha$-section.
    }
\end{figure}
A possible structure of $A_{x}([0,\e] \,|\, \alpha \leq \beta)$, along
with the above defined quantities, is shown in Figure
\ref{fig:alpha_sec}. We now see that
the range of local scales at all the $\alpha$-sections (below $\beta$)
is determined by the quantity
$$
\tau_{x}(\e,\beta) =
\RR_{x}(\e,\beta)-\rr_{x}(\e,\beta)-\dd_{x}(\e,\beta),
$$
which is the length of the side of the largest right isosceles triangle
whose interior lies completely in $A_{x}([0,\e]\,|\,\alpha\leq\beta)$.
As we mentioned in our earlier discussion, we would like
$\tau_{x}(\e,\beta)$ to be sufficiently large for appropriately chosen
$\e$ and $\beta$, as dictated by Lemma \ref{lem:interleave}. Recalling
quantities $f_c(a,b)$ and $g_c(a,b)$ from \eqref{nom:func}, we let
$f(a,b)$ and $g(a,b)$ be continuous functions which decrease to zero as $a,
b$ $\to$ $0$. We define
\begin{equation}\label{eq:def_E}
E_{x}(f,g) = \left\{
    \e\,\left|\,
    \begin{aligned}
        &\tau_{x}(\e', \beta)>\gamma ,\,
        \e' = \alpha_{x}^{l}(g(0,\e)),\\
        &\beta = g(\e',\e),\,
        \gamma=f(0,\e)+f(\e',\e)
    \end{aligned}
    \right.
    \right\}
\end{equation}

The set $E_x(f,g)$ and quantities $\tau_{x}(\e,\beta)$,
$\RR_{x}(\e,\beta)$, $\rr_{x}(\e,\beta)$, and $\dd_{x}(\e,\beta)$ have
some useful properties. We list them in the following lemma, whose proof is
given in the Appendix.
\begin{lemma}\label{lem:E_prop}\leavevmode
    Suppose that $A_{x}(\e)$ is weakly seemly. Then the following holds:
    \begin{enumerate}
        \item If $\e_1\leq\e_2,$ and $\beta_1\leq\beta_2$ then
            $$
            \begin{aligned}
                \rr_{x}(\e_1,\beta_1)&\leq\rr_{x}(\e_2, \beta_2),\quad
                \RR_{x}(\e_1,\beta_1)\geq\RR_{x}(\e_2, \beta_2),\\
                \dd_{x}(\e_1,\beta_1)&\leq\dd_{x}(\e_2, \beta_2),\quad
                \tau_{x}(\e_1,\beta_1)\leq\tau_{x}(\e_2, \beta_2).
            \end{aligned}
            $$
        \item $\dd_{x}(\e,\beta)\to 0$ as $\e\to 0$ for all
            $\beta\in(0,\alpha_{x}^{u}(0))$.
            Also, $\rr_{x}(\e,\beta)\to 0$ as $\e,\beta\to 0$.
        \item $E_x(f,g)$ is a non-empty interval.
    \end{enumerate}
\end{lemma}
It is also useful to note that if $\e\in E_{x}(f,g)$ then
$
\Delta_{x}(\e', \beta, \dd_{x}(\e',\beta)+\gamma)\neq\emptyset,
$
where we retained the notation from \eqref{eq:def_E}.

Keeping in mind the assumptions and notation stated at the beginning of
the section, we can now prove the following theorem.
\begin{theorem}\label{thm:lh_recover}
    Suppose that $K$ is weakly homologically seemly, 
    $\e\in E_{x}(f_c,g_c)$, and $c(\e'+\e)<\conv(P)$, where
    $\e'=\alpha_{x}^{l}(g_c(0,\e))$. Let
    \begin{equation*}
        \beta = g_{c}(\e',\e),\quad \gamma=f_{c}(0,\e)+f_{c}(\e',\e).
    \end{equation*}
    Then for all
    $$
        (R',r') \in \Delta_{x}(\e', \beta, \dd_{x}(\e',\beta)+\gamma),
    $$
    and
    $$
        R=R'-g_1(0,\e),\quad r=r'+g_c(\e',\e),
    $$
    we have
    \begin{align*}
        \im\big( C(\e,R)\to C(\e'+\e, r) \big)&\approx\h(K,K-\{x\}),\quad c=1\\
        \im\big( V(\e,R)\to V(\e'+\e, r) \big)&\approx\h(K,K-\{x\}),\quad c=s
    \end{align*}
\end{theorem}
\begin{proof}
    The conditions of the lemma imply that all $\alpha$-sections
    below $\beta$ in any $A_{x}(\tilde{\e})$, $\tilde{\e}\in[0,\e']$,
    contain $\Delta_x(\e',\beta, \dd_{x}(\e',\beta))$, and
    the side length of this triangle is greater than $\gamma$.
    In particular,
    \begin{align*}
        &(R', R'-f_c(0,\e))\in A_{x}(0\,|\,\alpha=g_c(0,\e)),\\
        &(R'-f_c(0,\e), R'-f_c(0,\e)-\dd_{x}(\e',\beta))\in
        A_{x}(g_c(0,\e)\,|\,\alpha=\e'),\\
        &(R'-f_c(0,\e)-\dd_{x}(\e',\beta), r')\in
        A_{x}(\e'\,|\,\alpha=\beta).
    \end{align*}

    We have the following inclusion induced homomorphisms:
    $$
    \begin{tikzcd}[row sep=small]
        F(g_c(0,\e), R'-f_c(0,\e))\ar[r, "\vphi_1"]&
        F(\e', R'-f_c(0,\e)-\dd_{x}(\e',\beta))\ar[d, "\vphi_2"]\\
        F(0,R')\ar[u, "\vphi_0"] & F(\beta, r')
    \end{tikzcd}
    $$
    By construction, restrictions $\im{\vphi_i}$ $\to$ $\im{\vphi_{i+1}}$,
    $i=0,1$, are isomorphisms, with $\im{\vphi_i} \approx
    \h(K,K-\{x\})$, $i=0,1,2$. Lemma \ref{lem:interleave} implies that
    the above sequence of homomorphisms factors as follows:
    \begin{equation}\label{eq:lh_diag}
        \begin{tikzcd}[column sep=small, row sep=small,
        /tikz/column 1/.append style={anchor=base west},
        /tikz/column 4/.append style={anchor=base east},
        /tikz/column 5/.append style={anchor=base east}]
        F(g_c(0,\e), R'-f_c(0,\e))\ar[rrr]&[-5em]&&[-5em]
        F(\e', R'-f_c(0,\e)-\dd_{x}(\e',\beta))\ar[d] \\
        L_c(\e, R)\ar[u] &F(0,R')\ar[l]&
        F(\beta, r') & L_c(\e'+\e, r)\ar[l]
    \end{tikzcd}
    \end{equation}
    where $L_1(a,b) = C(a,b)$ and $L_s(a,b) = V(a,b)$. Therefore, by
    Lemma \ref{lem:homob}
    $$\im(L_c(\e,R) \to L_c(\e'+\e, r)) \approx \h(K,K-\{x\}),$$
    thus yielding the required result.
\end{proof}
Thus, as long as $K$ is weakly homologically seemly at $x\in K$, the
local homology at $x$ can be recovered from the relative homology of
\v{C}ech or Vietoris-Rips complexes of an $\e$-sample, for $\e$
sufficiently small. The range of appropriate global and local scales are
essentially determined by the behavior of the functions
$\alpha_{x}^{l}(\cdot)$ and $\tau_{x}(\cdot, \cdot)$. Indeed, if
$\tau_{x}(\e,\beta)>0$ for some $\e>0$, $\beta\in[\alpha_{x}^{l}(\e),
\alpha_{x}^{u}(\e)]$, and if we can estimate how fast $\alpha_{x}(\cdot)$
decreases and $\tau_{x}(\cdot,\beta)$ increases as $\e\to 0$, then we
can determine how much $\e$ needs to be reduced so that the conditions
of Theorem \ref{thm:lh_recover} can be satisfied. This behavior
significantly simplifies if the ``gaps'' $\dd_{x}(\e,\beta)$
and $\alpha_{x}^{l}(\e)-\e$ are zero.
\begin{definition}
    The collection of sets $A_{x}(\e)$, $\e\geq0$, is called moderately
    seemly if it is weakly seemly and there are non-negative functions
    $\rho^{u}_{x}(\e,\alpha)$, $\rho^{l}_{x}(\e,\alpha)$ defined for
    $\e\in[0,\e_{x}]$, $\alpha\in[\alpha_{x}^{l}(\e),\alpha_{x}^{u}(\e)]$,
    such that the following conditions are satisfied:
    \begin{enumerate}
        \item $\rho^{u}_{x}$ and $\rho^{l}_{x}$ are, respectively,
            non-increasing and non-decreasing with respect to $\e$ and
            $\alpha$;
    \item $\rho^{l}_{x}(\e,\alpha)<\rho^{u}_{x}(\e,\alpha)$;
    \item $\rho^{l}_{x}(\e,\alpha)\searrow\rr_{x}(0,\alpha)$ and
        $\rho^{u}_{x}(\e,\alpha)\nearrow\RR_{x}(0,\alpha)$ as $\e\to 0$;
    \item for all
        $\rho\in[\rho^{l}_{x}(\e,\alpha),\rho^{u}_{x}(\e,\alpha)]$,
            $\e>0$, we have $(\rho,\rho,\alpha)\in A_{x}(\e)$.
    \end{enumerate}
    If, in addition, $\alpha_{x}^{u}(\e)=\e_x$, $\alpha_{x}^{l}(\e)=\e$,
    $\e\in[0,\e_x]$, then we call $A_{x}(\e)$ strongly seemly.  In
    either case, we also say that $K$ is moderately (respectively, strongly)
    homologically seemly at $x$.
\end{definition}
For a moderately seemly $A_{x}(\e)$, each $\alpha$-section of
$$A_{x}(\e\,|\, \alpha\in[\alpha_{x}^{l}(\e),\alpha_{x}^{u}(\e)])$$
contains a right isoseles triangle whose hypotenuse is the diagonal line
segment connecting
$$
(\rho^{l}_{x}(\e,\alpha), \rho^{l}_{x}(\e,\alpha))\text{ and }
(\rho^{u}_{x}(\e,\alpha), \rho^{u}_{x}(\e,\alpha)),
$$
and these triangles expand to the corresponding $\alpha$-sections of $A_{x}(0)$ as
$\e\to 0$. In this case,
we redefine
$$
\rr_{x}(\e,\beta) = \rho_{x}^{l}(\e,\beta),
\RR_{x}(\e,\beta)=\rho^{u}_{x}(\e,\beta),
$$
which implies
that $\dd_{x}(\e,\beta)=0$ for
$\e\in[0,\e_x]$ and $\beta\in[\alpha_{x}^{l}(\e), \alpha_{x}^{u}(\e)]$.
Hence, the quantity $\tau_{x}(\e,\beta)$ is simply the length of the
interval $(\rr_{x}(\e,\beta), \RR_{x}(\e,\beta))$, and the set
$E_{x}(f,g)$ is defined to contain $\e$ for which such intervals are
appropriately large at the required global scales. While these global
scales are still affected by the behavior of $\alpha_{x}^{l}(\cdot)$, it
is important that the range of local scales in Theorem
\ref{thm:lh_recover} is no longer restricted by $\dd_{x}(\e,\beta)$.

Looking back at the diagram \eqref{eq:def_A}, we can deduce that strong
seemliness of $A_{x}(\e)$ implies that for $\e\in[0,\e_x]$ the set
$A_{x}(\e\,|\,\alpha\in[\e,\e_x])$ consists
of triples $(R,r,\alpha)$ such that
$$
\h(D_{\alpha}(K), D_{\alpha}(K) - B_{\rho}(x)) \approx \h(K,
K-\{x\})
$$
for all $\rho\in[R,r]$. In this case, the (redefined)
quantities $\RR_{x}(\e,\beta)$ and
$\rr_{x}(\e,\beta)$ no longer depend on the first argument, so
$$
\RR_{x}(\e,\beta) = \RR_{x}(\beta),\quad
\rr_{x}(\e,\beta) = \rr_{x}(\beta),\quad
\tau_{x}(\e,\beta) = \tau_{x}(\beta).
$$
The latter
is now the only quantity that dictates appropriate global and local
scales, and it is the length of the interval consisting of local scales
$\rho>\beta$ such that
$$
\h(D_{\beta}(K), D_{\beta}(K) - B_{\rho}(x)) \approx \h(K, K-\{x\}).
$$

The definition of the set $E_{x}(f,g)$ simplifies:
$$
E_{x}(f,g) = \{\e\,|\,\tau_{x}(g(g(0,\e),\e))>f(0,\e)+f(g(0,\e), \e)\}
$$
Recalling notation from \eqref{nom:func}, we can see that
$$
    \begin{array}{llll}
        g_c(0,\e) &= (c+t)\e,&\quad g_c(g_c(0,\e),\e) &= (c^2+(t+1)c+t)\e,\\
        f_c(0,\e) &= (c+2t+1)\e,&\quad f_c(g_c(0,\e),\e) &= (c^2+(t+2)c+3t+1)\e.
    \end{array}
$$
Then $\e\in E_{x}(f_c,g_c)$ iff
$$
\tau_{x}((c^2+(t+1)c+t)\e) > (c^2+(t+3)c+5t+2)\e.
$$
Thus, in the case of strong seemliness, the statement of Theorem
\ref{thm:lh_recover} can be strengthened.
\begin{corollary}\label{cor:lh_strong_recover}
    Suppose that $K$ is strongly homologically seemly, and
    $\e$ is such that $c(1+c+t)\e<\conv(P)$
    and $\tau_{x}(\beta)>\gamma$, where
    $$
    \beta=(c^2+(t+1)c+t)\e,\quad \gamma=(c^2+(t+3)c+5t+2)\e.
    $$
    Then for all
    $$
    R'\in (\rr_{x}(\beta)+\gamma, \RR_{x}(\beta)),\quad
    r'\in(\rr(\beta), R'-\gamma),
    $$
    and
    $$
        R=R'-(1+t)\e,\quad
        r=r'+\beta,
    $$
    we have
    \begin{align*}
        \im\big( C(\e,R)\to C((1+c+t)\e, r) \big)&\approx\h(K,K-\{x\}),\quad c=1\\
        \im\big( V(\e,R)\to V((1+c+t)\e, r) \big)&\approx\h(K,K-\{x\}),\quad c=s
    \end{align*}
\end{corollary}
We can make the result even more concrete if we understand the behavior
of $\tau_{x}(\beta)$. Note that
$$
\tau_{x}(\beta) = \RR_{x}(\beta)-\rr_{x}(\beta) \geq
\RR_{x}(\e_x)-\rr_{x}(\beta) = M_0-\rr_{x}(\beta),
$$
where we let $M_0=\RR_{x}(\e_x)$ for convenience.
Hence, if $\rr_{x}(\beta)\leq M\beta^{m}$, for some $M>0$ and $\beta\in[0,\e_x]$, condition
$\tau_{x}(\beta)>\gamma$ will be satisfied if $M_0-M\beta^m>\gamma$.

\begin{corollary}\label{cor:lh_strong_simple}
    Keeping the assumptions and notation of Corollary
    \ref{cor:lh_strong_recover}, assume, in addition, that
    $\rr_{x}(\beta)\leq M\beta^{m}$ for $\beta\in[0,\e_x]$, where
    $M,m>0$. If $\e>0$ is such that
    $$
        M(c^2+(t+1)c+t)^m\e^m+(c^2+(3+t)c+5t+2)\e<M_0=\RR_{x}(\e_x),
    $$
    then the result of Corollary \ref{cor:lh_strong_recover} holds for
    \begin{align*}
        &R'\in\big(M(c^2+(t+1)c+t)^m\e^m+(c^2+(3+t)c+5t+2)\e, M_0\big),\\
        &r'\in\big(M(c^2+(t+1)c+t)^m\e^m, R'-(c^2+(3+t)c+5t+2)\e\big).
    \end{align*}
\end{corollary}

In general, estimating the behavior of $\rr_{x}(\beta)$ can be extremely
challenging. However, it is feasible for simple cases, like the example
in Figure \ref{fig:num_example}.  In particular, taking the point $x$ to
be a $0$-dimensional stratum of the example, one can obtain a (rough)
estimate: $\rr_{x}(\beta) \leq \sqrt{3}\beta$, for sufficiently
small $\beta$.

\subsection{Homological seemliness of a set}
Given a set $L\subseteq K$, we define $A_{L}(\e) = \bigcap_{x\in
L}{A_{x}(\e)}$. Since properties of Lemma \ref{lem:A} are preserved
under intersections, the concept of seemliness as well as all the
related quantities, which we defined for the case $L=\{x\}$, can be
readily extended to the case of any $L\subseteq K$ by formally
substituting $L$ for $x$ in the corresponding definitions.  However,
even for simple stratified sets, for example the set in Figure
\ref{fig:num_example}, whose stratification consists of the two
endpoints of the chord, the open chord, and the two open arcs, we may
have $A_{X}(0)=\emptyset$ for at least some $X\in\st$.  One can notice
that the culprit behind the trouble is the boundary of $X$, $\partial X
= \cl{X}-X$, since $A_{x}(0)$ decreases as $x$ gets closer to $\partial
X$. Hence, given $X\in\st$, we take a small $w>0$ and denote $X^{w}$ $=$
$X-B_{w}(\partial X)$.
\begin{lemma}\label{lem:A_0_X}
    Given a small enough $w>0$, we have
    $A_{X^{w}}(0\,|\,\alpha=\alpha')\neq\emptyset$ for all
    sufficiently small $\alpha'>0$, and these $\alpha$-sections have the
    structure described in Lemma \ref{lem:A_0}.
\end{lemma}
This lemma, whose proof can be found in the Appendix, implies that we
could recover local homology at each $x\in X^{w}$ from
an $\e$-sample if $A_{X^{w}}(\e)$ is at least weakly seemly.

\begin{definition}
    We say that $K$ is weakly (moderately, strongly) locally homologically
    seemly if for each $X\in\st$ the collection of sets
    $A_{X^{w}}(\e)$, $\e\geq 0$, is weakly (resp. moderately, strongly)
    seemly for all sufficiently small $w>0$.
\end{definition}

Note that due to compactness of $K$, its stratification, $\st$, is
finite. Let $h$ be the maximal height of the strata in $\st$, $h =
\max_{X\in\st}{\height{X}}$, and let $n_i=\card{\st_i}$, where $\st_i$
denotes the strata of height $i$. Imposing an arbitrary total order on
$\st_i$, $i=0,\ldots,h$, let $X_{ij}$ denote the
$j$-th stratum of $\st_i$, $1\leq j\leq n_i$. Let
$\bm{w}^i$ $=$ $(w_1^i,\ldots,w_{n_i}^i)\in\R^{n_i}_+$,
and $\bm{w}$ $=$ $(\bm{w}^0,$ $\ldots,$ $\bm{w}^{h-1})$ $\in$
$\R^{n}_{+}$, where $n=n_0+\cdots+n_{h-1}$. We denote
$$
X_{ij}^{\bm{w}} = X_{ij} -
\bigcup_{k=0}^{i-1}\bigcup_{m=1}^{n_k}{B_{w_m^k}(X_{km})},\quad
K^{\bm{w}} =
\bigcup_{i=0}^{h}\bigcup_{j=1}^{n_i}{X_{ij}^{\bm{w}}}
$$
It is useful to note that for sufficiently small $\bm{w}$ the sets
$X_{ij}^{\bm{w}}$ are non-empty and pairwise disjoint.
\begin{lemma}\label{lem:seemly_K}
    If $K$ is weakly (moderately, strongly) locally homologically
    seemly then $A_{K^{\bm{w}}}(\e)$ is weakly (resp. moderately,
    strongly) seemly for all sufficiently small $\bm{w}\in\R^{n}_+$.
\end{lemma}
The proof of the lemma is given in the Appendix. It follows that for a
sufficiently small $\e>0$ the set
$$
\Omega(\e,f,g) = \{\bm{w}\in\R^n_+\,|\, \e\in E_{K^{\bm{w}}}(f,g)\}
$$
is non-empty (where $f,g$ are continuous functions, as before). Moreover,
$$
\inf\{\|\bm{w}\|\,|\,\bm{w}\in\Omega(\e,f,g)\}\to 0,\text{ as } \e\to
0.
$$
We can now re-phrase the statement of Theorem
\ref{thm:lh_recover} (still keeping in mind all the assumptions and
notation stated at the beginning of the section).
\begin{theorem}\label{thm:lh_recover_K}
    Suppose $K$ is weakly homologically seemly, and
    $\e$ is small enough so that
    $\Omega(\e,f_c,g_c)\neq\emptyset$. Let
    $\bm{w}\in\Omega(\e,f_c,g_c)$, and assume that $c(\e'+\e)<\conv(P)$, where
    $\e'=\alpha_{K^{\bm{w}}}^{l}(g_c(0,\e))$. Denote
    \begin{equation*}
        \beta = g_c(\e',\e),\quad \gamma=f_c(0,\e)+f_c(\e',\e).
    \end{equation*}
    Take any
    $$
    (R',r') \in \Delta_{K^{\bm{w}}}(\e', \beta, \dd_{K^{\bm{w}}}(\e',\beta)+\gamma),
    $$
    and let
    $$
        R=R'-g_1(0,\e),\quad r=r'+g_c(\e',\e).
    $$
    Assuming that $x\in K^{\bm{w}}$, we have
    \begin{align*}
        \im\big( C(\e,R)\to C(\e'+\e, r) \big)&\approx\h(K,K-\{x\}),\quad c=1\\
        \im\big( V(\e,R)\to V(\e'+\e, r) \big)&\approx\h(K,K-\{x\}),\quad c=s
    \end{align*}
\end{theorem}
Thus, for a sufficiently small $\e>0$, we can guarantee local homology
recovery at all points of $P$ except those lying outside of
$D_{\e}(K^{\bm{w}})$, where $\bm{w}\in\R^n_+$ tends to $0$ as
$\e\to 0$. Corollary \ref{cor:intro} is a direct consequence of this
theorem and Theorem \ref{thm:weak_seemly} from Section \ref{sec:seemly_strat}.

As in the case of a singleton, the result can be made more specific if
we assume moderate or strong seemliness of $K$. We leave the details as an
exercise for the reader. The question that we
still need to answer is whether, under the assumptions discussed in
Section \ref{sec:prelim}, $K$ is weakly, moderately, or strongly
locally homologically seemly.

\section{Homological seemliness of certain classes of stratified sets}
\label{sec:seemly_strat}
In this section, we show that all of the stratified sets discussed in
Section \ref{sec:prelim} are indeed locally homologically seemly. We start with the most
general case.

\begin{theorem}\label{thm:weak_seemly}
    Let $K\subseteq\X$ be a compact homology stratified neighborhood
    retract. Then $K$ is weakly locally homologically seemly.
\end{theorem}
\begin{proof}
    The proof adapts some of the standard techniques used when dealing
    with Euclidean neighborhood retracts \citep[see e.g.][]{bredon2013,dold1995}.
    The underlying idea is that a sufficiently small neighborhood of $K$ strongly
    deforms onto $K$ through a slightly larger neighborhood, and such a
    deformation keeps exteriors of small, but not too small balls
    inside exteriors of slightly smaller balls. These facts then allow
    us to choose appropriate global and local scales.

    Since $K$ is a neighborhood retract
    and $\X$ is locally strongly convex, for any
    sufficiently small $\alpha>0$ there is $\e>0$ such that the
    neighborhood $D_{\e}(K)$ strongly deforms to $K$ through
    $B_{\alpha}(K)$.  Indeed, let $U$ be a neighborhood of $K$ such that
    there is a retraction $\pi:U\to K$. Due to compactness of $K$, we have
    $B_{\alpha'}(K)\subseteq U$ for a sufficiently small $\alpha'$.
    Moreover, due to continuity of $\pi$, $\alpha'$ can be chosen so small
    that there is a unique shortest path between any $x\in
    B_{\alpha'}(K)$ and $\pi(x)$. That is, we have
    $$
    \gamma_{x}:[0,1]\to\X,\quad \gamma_{x}(0)=x,\quad
    \gamma_{x}(1) = \pi(x).$$
    Consider a deformation
    $$
        G:B_{\alpha'}(K)\times[0,1]\to\X,\quad
        G(x,t) = \gamma_x(t).
    $$
    For any $\alpha\leq\alpha'$, preimage
    $G^{-1}(B_{\alpha}(K))$ is an open set containing $K\times[0,1]$.
    Since $[0,1]$ is compact, $G^{-1}(B_{\alpha}(K))$ contains an open
    set of the form $V\times[0,1]$, where $V\supseteq K$.
    Therefore it also contains
    $D_{\e}(K)\times[0,1]$ for $\e>0$ sufficiently small.

    Take $\rho_u>0$, $\beta>0$, and $\alpha>0$. For any $x\in K$ and
    any $\rho'\in[0,\rho_u]$, the preimage
    $$
        U_{\rho'}(x)=G^{-1}(B_{\alpha}(K)-D_{\rho'}(x))
    $$
    is an open set containing
    $$
        (K-B_{\rho'+\beta}(x))\times[0,1],
    $$
    and hence containing
    $$
        (D_{\e}(K)-B_{\rho'+\beta}(x))\times[0,1]
    $$
    for some $\e=\e(x,\rho',\beta)>0$. An argument analogous to that of Lemma
    \ref{lem:A_0_X} shows that due to compactness of $K$ and
    $[0,\rho_u]$ we can find  $\bar{\e}(\beta)>0$ such that for
    $\e\in(0,\bar{\e}(\beta))$ we have
    $$
        U_{\rho'}(x)\supseteq (D_{\e}(K)-B_{\rho''}(x))\times[0,1]
    $$
    for all $x\in K$, $\rho'\in [0,\rho_u]$, $\rho''\in[\rho'+\beta,\rho_u+\beta]$.
    
    In addition to our typical inclusions, we have the following maps
    $$
        D_{\e'}(K)\times\{0\} \to D_{\alpha}(K),\quad \e'\in[0,\e],
    $$
    and
    $$
        K \to D_{\e}(K)\times\{\ell\},\quad \ell=0,1,
    $$
    defined by $(y,0)\mapsto y$, and $y\mapsto (y,\ell)$,
    respectively.
    We also have a retraction map
    $$
        \pi:D_{\e}(K)\times\{1\} \to K,
    $$
    defined by $\pi(y,1)=G(y,1)$. The maps $G$ and $\pi$ are continuous, and
    the sets $D_{a}(K)$, $D_{a}(K) - B_{b}(x)$, with $b>a\geq0$, are
    compact. Therefore, for any $\rho'\in[0,\rho_u)$ we can find
    $\rho\in(\rho', \min\{\rho'+\beta, \rho_u\})$
    close enough to $\rho'$ so that
    \begin{equation*}
        \begin{aligned}
            &G((D_{\e}(K))\times[0,1]) \subseteq D_{\alpha}(K),\\
            &G((D_{\e}(K)-B_{\rho+\beta}(x))\times[0,1]) \subseteq
                D_{\alpha}(K)-B_{\rho}(x),\\
            &\pi((D_{\e}(K)-B_{\rho+\beta}(x))\times\{1\}) \subseteq
                K-B_{\rho}(x).
        \end{aligned}
    \end{equation*}
    Taking $\rho''$ $\in$ $[\rho+\beta,\rho_u+\beta]$,
    recalling the meaning of $F(a,b)$ from \eqref{nom:hom},
    and denoting
    \begin{equation*}
        F(a,b,I) = \h(D_{a}(K) \times I, (D_{a}(K)-B_{b}(x)) \times I),
    \end{equation*}
    where $a,b\in\R_+$, $I\subseteq[0,1]$,
    we obtain the following commutative diagram:
    \begin{equation}\label{eq:dgm_proof_ws}
    \begin{tikzcd}[row sep=small]
        F(0,\rho'')\ar[r]\ar[dd, equal]&
        F(\e,\rho'',\{1\})\ar[r, "\pi_*"]\ar[d, "\approx"] &
        F(0,\rho)\ar[d]\\
        &F(\e,\rho'',[0,1])\ar[r, "G_*"]&
        F(\alpha,\rho)\\
        F(0,\rho'')\ar[r]&
        F(\e,\rho'',\{0\})\ar[u, "\approx"]\ar[ur] &
    \end{tikzcd}
    \end{equation}
    Let $X\in\st$ be a stratum of $K$, and suppose $w>0$ is sufficiently
    small, so that $A_{X^{w}}(0)$ is non-empty.
    Let $\alpha_0>0$ be such that $A_{X^{w}}(0\,|\,$ $\alpha=\alpha_0)$ has
    a non-empty interior. For each $\alpha\in[0,\alpha_0]$ we have
    functions
    \begin{equation*}
        \begin{aligned}
            \RR_{X^{w}}(\alpha) &= \sup\{\rho\,|\,
                (\rho,\rho,\alpha) \in A_{X^{w}}(0)\},\\
            \rr_{X^{w}}(\alpha) &= \inf\{\rho\,|\, (\rho,\rho,\alpha)
                \in A_{X^{w}}(0)\}
        \end{aligned}
    \end{equation*}
    which are
    non-increasing and non-decreasing, respectively.
    Diagram \eqref{eq:dgm_proof_ws}, combined with diagrams
    \eqref{eq:def_A} and \eqref{eq:dgm_proof_A_0}, shows that
    if $(\rho''$, $\rho$, $\alpha)$ $\in$
    $A_{X^{w}}(0)$ for some $\rho''$ $\in$ $[\rho+\beta,\rho_u+\beta]$, then
    $(\rho'',\rho,\alpha)$ $\in$ $A_{X^{w}}(\e)$ for all
    $\rho''$ $\in$ $[\rho+\beta,\RR_{X^{w}}(\alpha))$.

    Let $\beta: [0,\alpha_0]\to\R_+$ be such that
    \begin{equation*}
        \beta(\alpha)\searrow 0, \text{ as } \alpha \to 0, \text{ and }
        \beta(\alpha_0) < \RR_{X^{w}}(\alpha_0) - \rr_{X^{w}}(\alpha_0).
    \end{equation*}
    According to the earlier discussion, for
    each $\alpha$ $\in$ $[0,\alpha_0]$ we can find
    $\e'(\alpha)$ $\in$ $(0,\bar{\e}(\alpha))$ such that
    taking
    \begin{equation*}
        \rho \in (\rr_{X^{w}}(\alpha),
        \RR_{X^{w}}(\alpha)-\beta(\alpha))\quad
        \text{ and }\quad
        \rho' \in [\rho+\beta(\alpha), \RR_{X^{w}}(\alpha))
    \end{equation*}
    we have
    \begin{equation*}
        (\rho', \rho, \alpha) \in A_{X^{w}}(\e'(\alpha)).
    \end{equation*}
    Note that we can choose $\e'(\alpha)$ to be non-decreasing with
    respect to $\alpha$.
    Indeed, $\RR_{X^{w}}(\alpha)$ is non-increasing and
    $\rr_{X^{w}}(\alpha)$ is non-decreasing, so having
    \begin{equation*}
        (\rho', \rho, \alpha_1) \in A_{X^{w}}(\e'(\alpha_1))
    \end{equation*}
    for
    \begin{equation*}
            \rho \in (\rr_{X^{w}}(\alpha_1),
            \RR_{X^{w}}(\alpha_1)-\beta(\alpha_1))\quad
            \text{ and }\quad
            \rho' \in [\rho+\beta(\alpha_1), \RR_{X^{w}}(\alpha_1))
    \end{equation*}
    implies 
    \begin{equation*}
        (\rho', \rho, \alpha_2) \in A_{X^{w}}(\e'(\alpha_1))
    \end{equation*}
    for
    \begin{equation*}
            \rho \in (\rr_{X^{w}}(\alpha_2),
            \RR_{X^{w}}(\alpha_2)-\beta(\alpha_2))\quad
            \text{ and }\quad
            \rho' \in [\rho+\beta(\alpha_2), \RR_{X^{w}}(\alpha_2)),
    \end{equation*}
    as long as $\alpha_1\leq\alpha_2$.
    If we define
    \begin{equation*}
        \e(\alpha) = \e'(\alpha) + \exp{(-1/\alpha)} \inf_{\alpha' \in [\alpha,\alpha_0]}
        (\bar{\e}(\alpha') - \e'(\alpha')),
    \end{equation*}
    we still have $\e(\alpha)$ $\in$ $(0,\bar{\e}(\alpha))$, and $\e(\alpha)$
    is strictly increasing with respect to $\alpha$.

    Let $\e_0=\e(\alpha_0)$. Define the function
    $\alpha_{X^{w}}^{u}(\cdot)$ on $[0,\e_0]$ to be the constant $\e_0$.
    Since $\e(\alpha)$ is strictly monotonic, it has
    the inverse (which is constant on intervals corresponding to
    discontinuities of $\e(\alpha)$). Hence, we define
    $\alpha_{X^{w}}^{l} = \e^{-1}$. By construction,
    \begin{equation*}
        d_H(A_{X^{w}}(0\,|\, \alpha=\alpha'), A_{X^{w}}(\e\,|\,
        \alpha=\alpha'))\leq \beta(\alpha^{l}_{X^{w}}(\e)),
    \end{equation*}
    for all $\alpha'\in[\alpha_{X^{w}}^{l}(\e), \e_0]$, thus yielding the
    needed result.

\end{proof}

Thus, Theorem \ref{thm:lh_recover_K} applies to $K$, and we see that
even for very general stratifies sets, local homology can be recovered
at a all points of an $\e$-sample, but for a small fraction, as long as
$\e$ is sufficiently small. One can expect that things only improve as
the stratified set becomes ``nicer''.

\begin{theorem}\label{thm:mod_seemly}
    Suppose that $\X$ is a complete Riemannian manifold, $K\subseteq\X$
    is a compact Whitney stratified set. Then $K$ is moderately locally
    homologically seemly.
\end{theorem}
\begin{proof}
    Similarly to the proof of Theorem \ref{thm:weak_seemly}, the idea is to use a strong
    deformation of a neighborhood of $K$ onto $K$ through a slightly larger
    neighborhood. But now we can employ the fact that $K$ is Whitney
    stratified, which will allow us to modify such a deformation so that
    the exterior and the interior of a small, but not too small ball are
    invariant with respect to the deformation.

    Since $K$ is Whitney stratified, there is a neighborhood $U\supseteq
    K$ and a strong deformation retraction $G:U\times[0,1]\to U$ of $U$
    onto $K$.  Due to compactness of $K$, we can find $\alpha'>0$ such
    that $B_{\alpha'}(K)\subseteq U$.

    Let $X\in\st$ be a stratum of $K$, and suppose $w>0$ is sufficiently
    small, so that $A_{X^{w}}(0)$ is non-empty.  Let
    $\alpha_0$ $\in$ $(0,\alpha']$ be such that
    $A_{X^{w}}(0\,|\,$ $\alpha=\alpha_0)$ has a non-empty interior. For
    each $\alpha$ $\in$ $[0,\alpha_0]$ we have functions
    \begin{equation*}
        \begin{aligned}
            \RR_{X^{w}}(\alpha) &= \sup\{\rho\,|\, (\rho,\rho,\alpha)
        \in A_{X^{w}}(0)\},\\
            \rr_{X^{w}}(\alpha) &= \inf\{\rho\,|\, (\rho,\rho,\alpha)
        \in A_{X^{w}}(0)\},
        \end{aligned}
    \end{equation*}
    which are non-increasing and non-decreasing,
    respectively.

    For any $x\in X^{w}$ and $\rho\in(0,\RR_{X^{w}}(0))$, the spheres
    $S_{\rho}(x)$ are transverse to all the strata of $K$. Therefore,
    $S_{\rho}(x)\cap K$ is Whitney a stratified set, and we can regard
    $S_{\rho}(x)$ itself as a Whitney stratified set with the strata of
    $S_{\rho}(x)\cap K$ and $S_{\rho}(x)-K$. All these stratified sets
    have the same structure, that is, for any $x_1,x_2\in X^{w}$,
    $\rho_1$, $\rho_2$ $\in$ $(0,\RR_{X^{w}}(0))$, there is a stratified
    homeomorphism $S_{\rho_1}(x_1)$ and $S_{\rho_2}(x)$, taking
    $S_{\rho_1}(x)\cap K$ to $S_{\rho_2}(x)\cap K$. Let $S$ denote one
    such stratified set, with $L\subset S$ homeomorphic to
    $S_{\rho}(x)\cap K$.

    For a given $x\in X^{w}$ and $\rho\in(0,\RR_{X^{w}}(0))$, using the
    smooth family of distance functions $d(y,\cdot)^2$ along with
    Thom's first isotopy lemma, we obtain a continuous map
    $J(y,\vrho,z)$, defined on an open subset of
    \begin{equation*}
    (B_{\delta}(x)\cap X) \times (\rho-\beta,\rho+\beta) \times
    (B_{\rho+\beta+\delta}(x) - B_{\rho-\beta-\delta}(x)),
    \end{equation*}
    for some $\delta$, $\beta>0$,  and mapping it into
    \begin{equation*}
        (B_{\delta}(x)\cap X) \times (\rho-\beta,\rho+\beta) \times S,
    \end{equation*}
    with $J(y,\vrho,\cdot):$ $S_{\vrho}(y)$ $\to$ $S$ a stratified
    homeomorphism smooth on each stratum. Since $L\subset S$ is Whitney stratified, we can find
    an open (in $S$) neighborhood $U_L$ of $L$ that strongly deformation
    retracts onto $L$. Then
    \begin{equation*}
    J^{-1}((B_{\delta}(x)\cap X) \times (\rho-\beta,\rho+\beta) \times U_L)
    \end{equation*}
    is an open neighborhood of
    \begin{equation*}
    (B_{\delta}(x)\cap X) \times (\rho-\beta,\rho+\beta) \times
        (S_{\rho}(x)\cap K),
    \end{equation*}
    and we can find $\delta'>0$, $\beta'>0$, and $\e>0$ such that for any
    $y\in (B_{\delta'}(x)\cap X)$ this neighborhood contains
    \begin{equation*}
        \{y\} \times [\rho-\beta',\rho+\beta'] \times (D_{\e}(K)\cap
        (D_{\rho+\beta'}(y) - B_{\rho-\beta'}(y))).
    \end{equation*}

    By pulling back the deformation retraction of $U_L$ and using
    partition of unity to glue it with $G$, we see that
    for each $x\in X^{w}$ and 
    $\rho$ $\in$ $(\rr_{X^{w}}(\alpha)$, $\RR_{X^{w}}(\alpha))$,
    with $\alpha$ $\in$ $(0,\alpha']$, we can
    find $\e=\e(x,\rho)>0$ such that $D_{\e}(K)$ strongly deforms onto $K$
    through $B_{\alpha}(K)$, and $S_{\rho}(x)$ is invariant under the
    deformation. The latter implies
    \begin{equation*}
        G_{x,\rho}((D_{\e}(K)-B_{\rho}(x)) \times[0,1]) \subseteq
        D_{\alpha}(K)-B_{\rho}(x),
    \end{equation*}
    where $G_{x,\rho}$ denotes the
    deformation. Moreover, such $\e(x,\rho)$ is bounded away from
    zero on a neighborhood of $(x,\rho)$, and hence on the compact
    $X^{w}$ $\times$ $[\rho_l,\rho_u]$, where
    $\rr_{X^{w}}(\alpha)<$ $\rho_l<$ $\rho_u<$ $\RR_{X^{w}}(\alpha)$.
    
    As in the proof of Theorem \ref{thm:weak_seemly},
    we have maps
    \begin{equation*}
        D_{\e'}(K)\times\{0\} \to D_{\alpha}(K),\quad \e'\in[0,\e],
    \end{equation*}
    and
    \begin{equation*}
        K \to D_{\e}(K)\times\{\ell\},\quad \ell=0,1,
    \end{equation*}
    as well as the retraction map:
    \begin{equation*}
        \pi_{x,\rho}:D_{\e}(K)\times\{1\} \to K,
    \end{equation*}
    defined by
    $\pi_{x,\rho}(y,1)$ $=$ $G_{x,\rho}(y,1)$. Hence, once again
    recalling the meaning of notation $F(a,b)$ from \eqref{nom:hom} and denoting
    \begin{equation*}
        F(a,b,I) = \h(D_{a}(K) \times I, (D_{a}(K)-B_{b}(x)) \times I),
    \end{equation*}
    where $a,b\in\R_+$, $I\subseteq[0,1]$,
    we obtain the following commutative diagram:
    \begin{equation}\label{eq:dgm_proof_ms}
    \begin{tikzcd}[row sep=small]
        F(0,\rho)\ar[r]\ar[dd, equal]&
        F(\e,\rho,\{1\})\ar[r, "(\pi_{x,\rho})_*"]\ar[d, "\approx"] &
        F(0,\rho)\ar[d]\\
        &F(\e,\rho,[0,1])\ar[r, "(G_{x,\rho})_*"]&
        F(\alpha,\rho)\\
        F(0,\rho)\ar[r]&
        F(\e,\rho,\{0\})\ar[u, "\approx"]\ar[ur] &
    \end{tikzcd}.
    \end{equation}
    This diagram shows that if $(\rho$, $\rho$, $\alpha)$ $\in$ $A_{X^{w}}(0)$ then
    $(\rho$, $\rho$, $\alpha)$ $\in$ $A_{X^{w}}(\e)$, for $\rho$ $\in$ $[\rho_l,\rho_u]$.
    
    Let $\beta: [0,\alpha_0]\to\R_+$ be such that
    \begin{equation*}
        \beta(\alpha)\searrow 0, \text{ as } \alpha\to 0, \text{ and }
        \beta(\alpha_0)<\RR_{X^{w}}(\alpha_0) - \rr_{X^{w}}(\alpha_0).
    \end{equation*}
    Based on the above discussion and proceeding as in Theorem
    \ref{thm:weak_seemly}, we obtain a strictly increasing function
    $\e:[0, \alpha_0]\to \R_+$, such that for any
    \begin{equation*}
        \alpha\in[0,\alpha_0] \text{ and }
        \rho \in [\rr_{X^{w}}(\alpha)+\beta(\alpha), \RR_{X^{w}}(\alpha)-\beta(\alpha)]
    \end{equation*}
    we have
    \begin{equation*}
        (\rho, \rho, \e(\alpha)) \in A_{X^{w}}(\e(\alpha)).
    \end{equation*}
    As in Theorem \ref{thm:weak_seemly}, we let
    $\e_0=\e(\alpha_0)$, define $\alpha_{X^{w}}^{u}(\cdot)$ on
    $[0,\e_0]$ to be the constant $\e_0$, and let
    $\alpha_{X^{w}}^{l}$ $=$ $\e^{-1}$.
    The claim of the theorem now readily follows by letting
   \begin{equation*}
        \rho^{l}_{X^{w}}(\alpha) = \rr_{X^{w}}(\alpha)+\beta(\alpha),
        \quad
        \rho^{u}_{X^{w}}(\alpha) = \RR_{X^{w}}(\alpha)-\beta(\alpha).
    \end{equation*}
\end{proof}

\begin{theorem}\label{thm:strong_seemly}
    Suppose that $\X$ is a complete Riemannian manifold, $K\subseteq\X$ is a
    compact Whitney stratified set with positive weak feature size.
    Then $K$ is strongly locally homologically seemly.
\end{theorem}
\begin{proof}
    As in the proof of Theorem \ref{thm:mod_seemly}, we shall use the
    fact that a small enough neighborhood of the form $D_{\alpha}(K)$
    strongly deforms onto $K$ through a slightly larger neighborhood
    (the one that actually deformation retracts onto $K$). But now we
    can also use the fact that $K$ has a positive weak feature size, so
    a small enough neighborhood of the form $D_{\alpha}(K)$ actually
    deformation retracts onto a smaller neighborhood $D_{\alpha'}(K)$,
    $\alpha'<\alpha$.

    So, let $U\supseteq K$ be a neighborhood of $K$ such that there is a
    strong deformation retraction $G:U\times[0,1]\to U$ of $U$ onto $K$.
    We can find $\alpha'>0$ such that $B_{\alpha'}(K)\subseteq U$,
    and we can assume $\alpha'<\wfs(K)$. Using the result in \cite{grove1993}, we can
    construct a gradient like vector field on $B_{\alpha'}(K)-K$  whose
    flow provides an isotopy between $S_{\alpha}(K)$ for
    $\alpha$ $\in$ $(0,\alpha']$. Take a stratum $X\in\st$ and
    choose $w>0$ as in Theorem \ref{thm:mod_seemly}. The claim of the theorem will follow if we
    find $\alpha_0$ $\in$ $(0,\alpha']$ such that for any
    $\alpha$ $\in$ $(0,\alpha_0]$, $x\in X^{w}$, and $\rho$ $\in$
    $[\rho_l,\rho_u]$, with
    $0<$ $\rho_l<$ $\rho_u<$ $\RR_{X^{w}}(0)$, we can modify this vector field to
    make the sphere $S_{\rho}(x)$ invariant with respect to the flow.
    Indeed, if this is true then we can deform $D_{\alpha}(K)$ until its
    image is inside $D_{\e(\alpha)}(K)$, where $\e(\alpha)$ is the
    function from Theorem \ref{thm:mod_seemly}, and then apply the
    deformation from Theorem \ref{thm:mod_seemly}.  Consequently, the
    concatenation of these two deformations allows us to choose
    $\e(\alpha)=\alpha$.

    For $y\in B_{\alpha'}(K)$, denote by $C_{y}$ be the set of its closest
    points in $K$, by $Q_{y}$ the set of directions from $y$ to $C_{y}$,
    and by $L_{y}$ the set of lines along these directions. Examining the
    construction of the gradient like vector field in \cite{grove1993}, we see
    that the needed modification of it is possible for
    $\alpha\in(0,\alpha']$ if the following conditions are satisfied:
    for any
    $x\in X^{w}$, $\rho\in [\rho_l,\rho_u]$, and $y\in (D_{\alpha}(K)-K)\cap
    S_{\rho}(x)$, the angle $\angle(n_{y},Q_{y}) \geq \theta$ for some
    $\theta>0$, where $n_{y}$ is the set of normal directions to
    $S_{\rho}(x)$ at $y$. To prove this, let us assume the opposite. Then we can find
    sequences
    \begin{equation*}
        \alpha_i \to 0,\quad \theta_i \to 0,\quad
        \rho_i \to \hat{\rho} \in[\rho_l,\rho_u],\quad
        x_i \to \hat{x} \in X^{w},
    \end{equation*}
    and a sequence
    \begin{equation*}
        y_i \in D_{\alpha_i}(K)\cap S_{\rho_i}(x_i),\quad
        y_i \to \hat{y} \in K\cap S_{\hat{\rho}}(\hat{x}),
    \end{equation*}
    such that
    $\angle(n_{y_i},Q_{y_i})$ $<$ $\theta_i$. Moreover, $C_{y_i}$ $\to$
    $\{\hat{y}\}$ $\subseteq$ $Z$, where $Z$ is a stratum of
    $S_{\hat{\rho}}(\hat{x})\cap K$.
    Using the family of
    stratified homeomorphisms from Theorem \ref{thm:mod_seemly}, and
    passing to a subsequence if necessary, we obtain sequences
    \begin{equation*}
        T_{z_i} S_{\hat{\rho}}(\hat{x})\to \tau \quad\text{ and }\quad
        l_i=\overline{z_i,\hat{y}}\to l,
    \end{equation*}
    where $z_i$ and $T_{z_i}S_{\hat{\rho}}(\hat{x})$ are the images of $y_i$,
    $T_{y_i} S_{\rho_i}(x_i)$ under the stratified homeomorphisms (which
    are smooth on each stratum). By construction, the distance (in the
    projective space) between
    the normal lines at $y_i$ and the set of lines $L_{y_i}$ goes to
    zero. Note that $C_{y_i}$ are the points of tangency between
    $S_{d_i}(y_i)$ and the strata $Y\geq X$, where $d_i=d(y_i, K)$. By passing to a
    subsequence if necessary, we get points $p_i$ $\in$ $C_{y_i}$ such that
    \begin{equation*}
        T_{p_i}Y \to \tau' \supseteq T_{\hat{y}}X \supseteq T_{\hat{y}}Z.
    \end{equation*}
    Consequently, the distance
    between the normal lines at $y_i$ and $\overline{y_i\hat{y}}$ goes to
    zero. By smoothness of our stratified homeomorphisms,
    this implies that the distance between the normal lines at
    $z_i$ and $l_i$ also goes to zero. Therefore, $l\not\subseteq\tau$,
    which contradicts Whitney condition (b).

\end{proof}

Let us now assume that $\X$ is a Euclidean space.
The \emph{reach} (also known as the minimum local feature
size) of a boundaryless submanifold $M\subset\X$, denoted $\reach(M)$, is defined as
the supremum over all numbers $\alpha>0$ such that the normal bundle of $M$
of radius $\alpha$ is embedded in $\X$ \citep[see e.g.][]{niyogi.etal2008}. In this case,
$S_{\alpha}(M)$ for $\alpha\in(0,\reach(M))$ are smooth submanifolds of
$\X$ of co-dimension $1$ (the boundary of the normal bundle of radius
$\alpha$).

\begin{lemma}\label{lem:reach_hom}
    Let $\alpha$ $\in$ $(0,\reach(M))$. Then for any $\rho$ $\in$
    $(\alpha$, $2\reach(M) - \alpha)$ and any $x$ $\in$ $M$ we have
    $$
    \h(D_{\alpha}(M), D_{\alpha}(M) - B_{\rho}(x))\approx \h(M,
    M-B_{\rho}(x))\approx \h(M, M-\{x\}).
    $$
\end{lemma}
\begin{proof}
    The sphere $S_{\rho}(x)$ is transverse to $M$ as well as to all
    $S_{\alpha'}(M)$, $\alpha'\in[0,\alpha]$. Suppose this is not
    the case. Then we have a point of tangency $p\in S_{\rho}(x)\cap
    S_{\alpha'}(M)$. Hence, the line $\overline{xp}$ is normal to
    $S_{\alpha'}(M)$ at $p$. Also,
    $p$ belongs to the normal ball at some $y\in M$. This implies that 
    the normal line at $y$ through $p$ coincides with $\overline{xp}$.
    The line segment $[xy]$ has length $\rho+\alpha'<2\reach(M)$. Then
    its midpoint $q=(x+y)/2$ has distance $d(q,y)=\rho_q<\reach(M)$, and
    therefore belongs to the closed normal ball at $y$ of radius $\rho_q$.
    On the other hand, since $d(q,x)=\rho_q$, we have a point $z\in M$
    closest to $q$ with $d(q,z)\leq \rho_q$. Thus, $\overline{zq}$ normal
    to $M$ at $z$, which implies that $q$ belongs to the closed normal ball
    at $z$ of radius $\rho_q$. But then a normal bundle of radius
    $\rho'\in(\rho_q,\reach(M))$ is not embedded in $\X$. Contradiction.

    This implies, in particular, that $D_{\rho}(x)\cap M$ is a closed
    (topological) ball, as it is a sublevel set of the smooth function $f_x(\cdot) =
    d(x,\cdot)^2$, and contains a single critical point, the minimum, at
    $x$. Consequently,
    \begin{equation*}
        \h(M, M-B_{\rho}(x))\approx \h(M,M-\{x\}).
    \end{equation*}

    Each point $p\in D_{\alpha'}(M)$ has a unique closest point $y\in
    M$. Moreover, the above transversality result implies that if $p\in
    S_{\rho}(x)$ then the angle
    between the lines $\overline{py}$ and $\overline{px}$ is bounded
    away from zero. Therefore, as in the proof of Theorem
    \ref{thm:strong_seemly}, we can construct a deformation retraction
    of $D_{\alpha'}(M)$ onto $M$ such that $S_{\rho}(x)$ stays
    invariant. This implies
    \begin{equation*}
        \h(D_{\alpha'}(M), D_{\alpha'}(M)-B_{\rho}(x))\approx \h(M, M-B_{\rho}(x)).
    \end{equation*}
\end{proof}

We shall now combine the above lemma with Corollary
\ref{cor:lh_strong_recover} to strengthen our local homology recovery
result for submanifolds. So, we assume
that $K\subset \R^n$ is a closed smooth submanifold with reach $\nu>0$,
and $P\subset \R^n$ is its $\e$-sample. As in Section \ref{sec:seemly},
we let $p\in P$ be an arbitrary point, and $x\in K$
be either a point with $d(x,p)<\e$, if $P$ is noisy, or the closest
point to $p$, if $P$ is noise-free. We recall notation
\eqref{nom:const}, \eqref{nom:hom}, and \eqref{nom:func}, introduced in Section
\ref{sec:prelim}, and define 
\begin{equation*}
    \beta=(c^2+(t+1)c+t),\quad \gamma=(c^2+(t+3)c+5t+2).
\end{equation*}
We can now state the following result.
\begin{theorem}\label{thm:manifold_seemly}
    With the assumptions and notation stated above, suppose that
    $\e<\dfrac{2\nu}{2\beta+\gamma}$. Take any
    \begin{equation*}
        R \in ((\beta+\gamma-t-1)\e, 2\nu-(\beta+t+1)\e),\quad
        r \in (2\beta\e, R-(\gamma-\beta)\e).
    \end{equation*}
    Then
    \begin{align*}
        \im\big( C(\e,R)\to C((1+c+t)\e, r) \big)&\approx\h(K,K-\{x\}),\quad c=1\\
        \im\big( V(\e,R)\to V((1+c+t)\e, r) \big)&\approx\h(K,K-\{x\}),\quad c=\sqrt{2}
    \end{align*} 
\end{theorem}
\begin{proof}
    The proof follows directly from Corollary
    \ref{cor:lh_strong_recover}, since Lemma \ref{lem:reach_hom} implies
    that for any $x\in M$ and $\alpha\in(0,\nu)$
    we have $\rr_{x}(\alpha)=\alpha$, $\RR_{x}(\alpha)\geq 2\nu-\alpha$,
    and hence the condition $\tau_x(\beta)>\gamma$ follows from
    $\e<\dfrac{2\nu}{2\beta+\gamma}$.
\end{proof}
The above theorem provides an improvement over the analogous result in
\cite{dey.etal2014}, since $2\beta+\gamma\leq 15+8\sqrt{2}$. We can also obtain a
similar result for manifolds with boundary.

\begin{corollary}
    Suppose that $K\subset \R^n$ is a smooth, compact submanifold with
    boundary. Let
    \begin{equation*}
        \nu=\min\{\reach(K), \reach(\partial K)\},
    \end{equation*}
    and keep the rest of
    the assumptions and notation of Theorem \ref{thm:manifold_seemly}.
    Then the result of Theorem \ref{thm:manifold_seemly} holds for
    $p\in P\cap B_{\e}(\partial K)$, with unchanged $R$ and $r$, as well
    as for $p\in P\cap B_{\e}(K^{w})$, where
    \begin{equation*}
        w\in((2\beta+\gamma)\e, 2\nu)
    \end{equation*}
    is such that $K^{w}=K-D_{w}(\partial K)\neq\emptyset$. In the
    latter case, we also impose the restriction $R<w$.
\end{corollary}
\begin{proof}
    As Theorem \ref{thm:manifold_seemly}, this result follows directly
    from Corollary \ref{cor:lh_strong_recover}. The condition $R<w$
    guarantees that the result of Lemma \ref{lem:reach_hom} holds.
\end{proof}
The Corollary provides an example when we can obtain a specific estimate on
the size of the region where recovery of local homology may not succeed.

\section{Conclusion}
\label{sec:conclusion}

By introducing a new concept of local homological seemliness, we showed
that local homology can be recovered even from fairly general homology
stratified sets, and it can be done using Vietoris-Rips complexes. We
also showed that the size of the region where recovery may not be
feasible decreases to zero as the sample becomes increasingly better. We
obtained a concrete bound on this size for the case of a smooth manifold
with boundary. It is feasible that one can obtain concrete bounds in
more general cases. In particular, it may be possible to show that if
$L\subset\R^n$ is a transverse union of finitely many smooth closed
manifolds, and $K\subseteq L$ is a closed stratified subset, then for
small enough $\e$ the size of the smallest $\bm{w}$ that still allows us
to recover local homology of $K^{\bm{w}}$ depends on $\e$ in a Lipschitz
way, i.e. $\|\bm{w}\|\leq C\e$ for some constant $C>0$. Obtaining
similar concrete bounds for general Whitney stratified sets is more challenging,
and is a possible direction of future research. In addition, we can try
to combine concrete size bounds on the region of
possible recovery failure with various sampling schemes to obtain concrete
estimates on the fractions of points where local homology recovery may
fail. 

Our results can also be used to develop an alternative algorithm for
assigning the points of an $\e$-sample $P$ to the corresponding strata
of $K$.  Indeed, for each point $p\in P$ there is another point $q\in P$
and a point $x\in K$ such that $d(x,p)<\e$, $d(x,q)<\e$.  Consulting the
proof of Theorem \ref{thm:lh_recover} and reusing that notation, we see
that if the conditions of the theorem are satisfied then it follows from
the diagram \eqref{eq:lh_diag} that the image of the homomorphism
$L_{c}^{p}(\e,R)\to L_{c}^{q}(\e'+\e,r)$ (where the superscript denotes
the point for which $L_c$ is constructed) also captures local homology
at $x$. This suggests that one can try to transitively group together
all points $p,q\in P$ with distance $d(p,q)<2\e$ for which the images of
$L_{c}^{p}(\e,R)$ $\to$ $L_{c}^{q}(\e'+\e,r)$ and $L_{c}^{q}(\e,R)$
$\to$ $L_{c}^{q}(\e'+\e,r)$ coincide. Of course, there are multiple caveats, as
described in \cite{bendich.etal2010}, and additional research in this
direction is needed.

\appendix
\section*{Appendix}
\begin{proof}[Proof of Lemma \ref{lem:trans_union}]\leavevmode
    It is known that (b)-regular stratified sets belong to a wider class
    of (c)-regular stratified sets introduced in
    \citep{bekka1991}. The definition of (c)-regularity is somewhat
    technical, and we refer the interested reader to the original paper
    by Bekka or to the book by Pflaum \citep[see][Section
    1.4.13]{pflaum2001}. It is also known that 
    transverse union of (c)-regular stratified sets with strata
    $\st$ and $\mathcal{R}$ is again a (c)-regular stratified set with
    the stratification $\st\cup_t\mathcal{R}$. Let us show that condition
    (b) is satisfied for $\st\cup_t\mathcal{R}$. Take
    $Z,W\in \st\cup_t\mathcal{R}$, with $Z\leq W$. Let $z\in Z$, and
    consider sequences
    \begin{equation*}
        z_i,w_i\to z, \text{ with } z_i\in Z, w_i\in W
    \end{equation*}
    such that
    \begin{equation*}
        T_{w_i}W\to \tau,\quad l_i=\overline{w_iz_i}\to l.
    \end{equation*}
    Since condition (b) holds for both $\st$ and
    $\mathcal{R}$, it holds if $Z\neq X\cap Y$ for any $X\in\st$,
    $Y\in\mathcal{R}$.
    Suppose
    \begin{equation*}
        Z=X_1\cap Y_1,\quad X_1\in\st,\quad Y_1\in\mathcal{R}.
    \end{equation*}
    Since condition (b) holds for transverse intersections, we only need
    to consider the case when $W=X_2-Y_2$ or $W=Y_2-X_2$ for some
    $X_2\in\st$, $Y_2\in\mathcal{R}$. Without
    loss of generality, we may assume the former, $W=X_2-Y_2$. Since
    $Z\leq W$, we have $Z\subseteq X_1\leq X_2$. So, $z_i\in X_1$,
    $w_i\in X_2$, and since condition (b) holds for $X_1\leq X_2$, we get
    $l\subseteq\tau$.
\end{proof}

The following proof employs the meaning of $F(a,b)$ from \eqref{nom:hom}.
\begin{proof}[Proof of Lemma \ref{lem:A_0}]\leavevmode
    It follows from the proof of Lemma \ref{lem:A_0_X} that for all
    sufficiently small $\e\geq 0$ we can find $\rho_u>\rho_l>0$
    such that 
    $(\rho,\rho,\e)\in A_{x}(0)$ for $\rho\in[\rho_l,\rho_u]$.
    Thus, $A_{x}(0\,|\,\alpha=\e)\neq\emptyset$.

    Now, suppose $\rho_u>\rho_l>0$ and $\e>0$ are such that
    \begin{equation*}
        (\rho_l,\rho_l,\e) \in A_{x}(0),\quad
        (\rho_u,\rho_u,\e) \in A_{x}(0).
    \end{equation*}
    Take $\rho_1, \rho_2$ $\in$ $[\rho_l,\rho_u]$, $\rho_1\leq\rho_2$,
    $\e'\in[0,\e]$.
    We have the following commutative diagram:
    \begin{equation}\label{eq:dgm_proof_A_0}
    \begin{tikzcd}[column sep=small, row sep=small]
        F(\e,\rho_u)\ar[r] & F(\e,\rho_2)\ar[r] &
        F(\e,\rho_1)\ar[r] & F(\e,\rho_l)\\
        F(\e',\rho_u)\ar[r]\ar[u] & F(\e',\rho_2)\ar[r]\ar[u] &
        F(\e',\rho_1)\ar[r]\ar[u] & F(\e',\rho_l)\ar[u]\\
        F(0,\rho_u)\ar[r]\ar[u] & F(0,\rho_2)\ar[r]\ar[u] &
        F(0,\rho_1)\ar[r]\ar[u] & F(0,\rho_l)\ar[u]
    \end{tikzcd}
    \end{equation}
    The bottom row consists of isomorphisms, and the maps from the
    bottom row to the top row along each column are injective, yielding
    $(\rho_2,\rho_1,\e')\in A_{x}(0)$. The diagram also implies
    that if $(\rho_u,\rho_l,\e)\in A_{x}(0)$ then
    $(\rho,\rho,\e')\in A_{x}(0)$ for all
    $\rho\in[\rho_l,\rho_u]$, $\e'\in[0,\e]$. Consequently, each
    $A_{x}(0\,|\,\alpha=\e')$ is a right isosceles triangle with the legs
    parallel to the axes and the hypotenuse lying on the diagonal, and
    $A_{x}(0\,|\,\alpha=\e)\subseteq A_{x}(0\,|\,\alpha=\e')$

    Note that the remark at the end of the proof of Lemma
    \ref{lem:A_0_X} implies that the vertical distance between the
    horizontal leg of our right triangle $A_{x}(0\,|\,\alpha=\e)$ and
    the $R$-axis tends to zero as $\e\to 0$.

\end{proof}

The following proof employs the meaning of $F(a,b)$ from \eqref{nom:hom}.
\begin{proof}[Proof of Lemma \ref{lem:A}]\leavevmode

    Part \eqref{lem:A_1}. Consider diagram \eqref{eq:def_A}
    and note that $\ii$ factors through $F(\delta, R)$
    for $\delta\in(0,\e)$. Therefore, we can
    replace $\e$ in diagram \eqref{eq:def_A} by any $\delta\in[0,\e]$ retaining
    all of the properties. Hence,
    \begin{equation*}
        (R,r,\alpha)\in A_{x}(\e) \implies
        (R,r,\alpha)\in A_{x}(\delta)
    \end{equation*}
    for all $\delta\in[0,\e]$.

    Part \eqref{lem:A_2}. Assume that each of
    the sets contains at least two elements (otherwise its a degenerate
    interval). So, let
    \begin{equation*}
        \begin{aligned}
            &\alpha_1, \alpha_2 \in A_{x}(\e | R=R', r=r'),\\
            &r_1, r_2 \in A_{x}(\e | R=R', \alpha=\alpha'),\\
            &R_1,R_2\in A_{x}(\e | r=r',\alpha=\alpha'),
        \end{aligned}
    \end{equation*}
    with $\alpha_1<\alpha_2$, $r_1<r_2$, $R_1<R_2$, and let
    \begin{equation*}
        \alpha\in(\alpha_1,\alpha_2),\quad
        r\in(r_1,r_2),\quad
        R\in(R_1,R_2).
    \end{equation*}
    Inclusions yield the following commutative diagrams:
    \begin{equation}\label{eq:lem_A_2_alpha}
    \begin{tikzcd}
        F(0,0)&
        F(0,R)\ar[l,"\approx"]\ar[r,"\vphi_0"]&
        F(\e,R) \ar[r, "\vphi_1"] &
        F(\alpha_1,r) \ar[r, "\vphi_2"] &
        F(\alpha,r) \ar[r, "\vphi_3"] &
        F(\alpha_2,r)
    \end{tikzcd}
    \end{equation}
    \begin{equation}\label{eq:lem_A_2_r}
    \begin{tikzcd}
        F(0,0)&
        F(0,R)\ar[l,"\approx"]\ar[r,"\phi_0"]&
        F(\e,R) \ar[r, "\phi_1"] &
        F(\alpha,r_2) \ar[r, "\phi_2"] &
        F(\alpha,r) \ar[r, "\phi_3"] &
        F(\alpha,r_1)
    \end{tikzcd}
    \end{equation}
    \begin{equation}\label{eq:lem_A_2_R}
    \begin{tikzcd}
        F(\e,R_2) \ar[r, "\psi_1"] &
        F(\e,R) \ar[r, "\psi_2"] &
        F(\e,R_1) \ar[r, "\psi_3"] &
        F(\alpha,r)\\
        F(0,R_2)\ar[r,"\approx"]\ar[u,"\vphi_1"]&
        F(0,R)\ar[r,"\approx"]\ar[u,"\vphi_2"]&
        F(0,R_1)\ar[r,"\approx"]\ar[u,"\vphi_3"]&
        F(0,0)
    \end{tikzcd}
    \end{equation}
    In diagram \eqref{eq:lem_A_2_alpha}, $\vphi_0$ is injective and restrictions
    $\im{\vphi_0}\to\im{\vphi_1}$ and
    $\im{\vphi_0}\to\im{\vphi_3\circ\vphi_2\circ\vphi_1}$ are
    isomorphisms. Therefore, we must also
    have $\im{\phi_0}\to\im{\vphi_2\circ\vphi_1}$ is an isomorphism. Thus,
    \begin{equation*}
        \alpha_1,\alpha_2\in A_{x}(\e | R=R', r=r') \implies
        \alpha\in A_{x}(x | R=R', r=r')
    \end{equation*}
    for all $\alpha\in[\alpha_1,\alpha_2]$. An analogous analysis of diagram
    \eqref{eq:lem_A_2_r} gives the result for the set
    $A_{x}(\e|R=R',\alpha=\alpha')$.
    In diagram \eqref{eq:lem_A_2_R}, $\vphi_3$ is
    injective, therefore $\psi_2\circ\vphi_2$ is injective, and hence
    $\vphi_2$ is injective. And since $\im{\vphi_3}\to\im{\psi_3}$ is an
    isomorphism, $\im{\vphi_2}\to\im{\psi_3\circ\psi_2}$ is
    also an isomorphism. Thus,
    \begin{equation*}
        R_1, R_2\in A_{x}(\e | r=r', \alpha=\alpha') \implies
        R\in A_{x}(\e | r=r',\alpha=\alpha')
    \end{equation*}
    for all $R\in[R_1, R_2]$.

    Part \eqref{lem:A_3}. Take $R_2>R_1\geq R_u\geq
    r_u\geq r_1>R_l\geq r_l$ and consider the following commutative
    diagram induced by inclusions.
    \begin{equation}\label{eq:lem_A_3}
        \begin{tikzcd}[column sep=tiny]
            &&
        F(\alpha_1,r_u)\ar[r] &
        F(\alpha_1,r_1)\ar[r] &
        F(\alpha_1,R_l)\ar[r] &
        F(\alpha_1,r_l)\\
        F(\e,R_2) \ar[r]&
        F(\e,R_1) \ar[r] &
        F(\e,R_u) \ar[u]\ar[rr]&&
        F(\e,R_l)\ar[u]&\\
        F(0,R_2)\ar[u]\ar[r] &
        F(0,R_1)\ar[u]\ar[r,"\approx"] &
        F(0,R_u)\ar[u]\ar[r,"\approx"] &
        F(0,R_l)\ar[r,"\approx"]\ar[ur]&
        F(0,0)&
    \end{tikzcd}
    \end{equation}
    To simplify our exposition, we shall slightly abuse notation and
    write $(a,b)\to (c,d)$ to indicate the homomorphism $F(a,b)\to
    F(c,d)$, with $[(a,b)(c,d)]=\im{( (a,b)\to (c,d) )}$ denoting the
    corresponding image. Note that since $(R_1,r_1,\alpha_1)\in A_{x}(\e)$,
    the homomorphism $(0,R_1)\to(\e,R_1)$ is injective and
    $[(0,R_1)(\e,R_1)]\to[(\e,R_1)(\alpha_1,r_1)]$ is an isomorphism.

    To show the first inclusion, we start by showing that
    \begin{equation*}
        R_2\notin A_{x}(\e|r=r_1,\alpha=\alpha_1)\implies
        A_{x}(\e|R=R_2,\alpha=\alpha_1)=\emptyset.
    \end{equation*}
    Note that the map $(0,R_2)\to(0,R_1)$ in diagram \eqref{eq:lem_A_3} cannot be an
    isomorphism in this case.  Indeed, if it is, then
    \begin{equation*}
        [(0,R_2)(\e,R_1)] \approx
        [(0,R_1)(\e,R_1)] \approx
        [(0,R_2)(\e,R_2)],
    \end{equation*}
    hence 
    \begin{equation*}
        [(0,R_2)(\e,R_2)] \to [(\e,R_2)(\alpha_1,r_1)]
    \end{equation*}
    is an
    isomorphism, implying $R_2\in A_{x}(\e|r=r_1,\alpha=\alpha_1)$, which is a
    contradiction. But if $(0,R_2)\to(0,R_1)$ is not an isomorphism
    then it follows directly from the definition of $A_{x}(\e)$ that
    $(R_2,r',\alpha')\notin A_{x}(\e)$ for any $r',\alpha'$.

    Now let
    \begin{equation*}
        R_2\in A_{x}(\e|r=r_1,\alpha_1),
    \end{equation*}
    and suppose that
    \begin{equation*}
        r_l, r_u \in A_{x}(\e|R=R_1,\alpha=\alpha_1).
    \end{equation*}
    Then in diagram \eqref{eq:lem_A_3} the maps
    \begin{equation*}
        [(0,R_1)(\e,R_1)] \to [(\e,R_1)(\alpha_1,r_i)],\quad i\in\{l,u\},
    \end{equation*}
    as well as the map
    \begin{equation*}
        [(0,R_2)(\e,R_2)] \to [(\e,R_2)(\alpha_1,r_1)],
    \end{equation*}
    are isomorphisms.
    Since $(0,R_2)$ $\to$ $(0,R_1)$ is an isomorphism in this case, we have
    \begin{equation*}
        [(0,R_2)(\e,R_1)] \approx
        [(0,R_1)(\e,R_1)] \approx
        [(0,R_2)(\e,R_2)],
    \end{equation*}
    which implies that
    \begin{equation*}
        [(0,R_2)(\e,R_2)] \to [(\e,R_2)(\alpha_1,r_i)],\quad
        i\in\{l,u\},
    \end{equation*}
    are isomorphisms. Thus,
    \begin{equation*}
        r_l,r_u \in A_{x}(\e|R=R_2,\alpha=\alpha_1).
    \end{equation*}

    For the second inclusion, we again start by showing
    \begin{equation*}
        r_l\notin A_{x}(\e|R=R_1,\alpha=\alpha_1) \implies
        A_{x}(\e|r=r_l,\alpha=\alpha_1)=\emptyset.
    \end{equation*}
    So, suppose that the map
    \begin{equation*}
        [(0,R_1)(\e,R_1)]\to [(\e,R_1)(\alpha_1,r_l)]
    \end{equation*}
    is not an isomorphism. Since
    \begin{equation*}
        [(0,R_1)(\e,R_1)]\to [(\e,R_1)(\alpha_1,r_1)]
    \end{equation*}
    is an isomorphism, this implies that
    \begin{equation*}
        [(\e,R_1)(\alpha_1,r_1)] \to [(\alpha_1,r_1)(\alpha_1,r_l)]
    \end{equation*}
    is not injective. Clearly, if
    \begin{equation*}
        R_2 \notin A_{x}(\e|r=r_1,\alpha=\alpha_1)
    \end{equation*}
    then
    \begin{equation*}
        [(0,R_2)(\e,R_2)] \to [(\e,R_2)(\alpha_1,r_1)]
    \end{equation*}
    is not an isomorphism, implying that
    \begin{equation*}
        [(0,R_2)(\e,R_2)] \to [(\e,R_2)(\alpha_1,r_l)]
    \end{equation*}
    is not an isomorphism, i.e.
    \begin{equation*}
        R_2 \notin A_{x}(\e|r=r_l,\alpha=\alpha_1).
    \end{equation*}
    So, assume
    \begin{equation*}
        R_2\in A_{x}(\e|r=r_1,\alpha=\alpha_1).
    \end{equation*}
    Then we see that maps
    \begin{equation*}
        [(0,R_2)(\e,R_2)] \to [(\e,R_2)(\alpha_1,r_l)]
        \quad\text{ and }\quad
        [(0,R_u)(\e,R_u)] \to [(\e,R_u)(\alpha_1,r_l)]
    \end{equation*}
    are not isomorphisms.
    Therefore,
    \begin{equation*}
        R_2, R_u \notin A_{x}(\e|r=r_l,\alpha=\alpha_1).
    \end{equation*}
    Now, assuming that
    \begin{equation*}
        [(0,R_l)(\e,R_l)] \to [(\e,R_l)(\alpha_1,r_l)]
    \end{equation*}
    is an isomorphism would imply that
    \begin{equation*}
        [(0,R_l)(\alpha_1,R_l)] \to [(\alpha_1,R_l)(\alpha_1,r_l)]
    \end{equation*}
    is injective, which is impossible since
    \begin{equation*}
        [(0,R_1)(\alpha_1,r_1)] \to [(\alpha_1,r_1)(\alpha_1,r_l)]
    \end{equation*}
    is not injective. Thus, we also have
    \begin{equation*}
        R_l \notin A_{x}(\e|r=r_l,\alpha=\alpha_1),
    \end{equation*}
    and so
    \begin{equation*}
        A_{x}(\e|r=r_l,\alpha=\alpha_1) = \emptyset.
    \end{equation*}
    
    Now let
    \begin{equation*}
        r_l \in A_{x}(\e|R=R_1,\alpha=\alpha_1),
    \end{equation*}
    and suppose that
    \begin{equation*}
        R_2, R_u \in A_{x}(\e|r=r_1,\alpha=\alpha_1).
    \end{equation*}
    Then in diagram \eqref{eq:lem_A_3} the maps
    \begin{equation*}
        [(0,R_i)(\e,R_i)] \to [(\e,R_i)(\alpha,r)],\quad
        i\in\{u,m\},
    \end{equation*}
    as well as the maps
    \begin{equation*}
        (0,R_2) \to (0,R_1)\quad\text{ and }\quad
        [(0,R_1)(\e,R_1)] \to [(\e,R_1)(\alpha_1,r_l)],
    \end{equation*}
    are isomorphisms. The latter implies that
    \begin{equation*}
        [(\e,R_1)(\alpha_1,r_1)] \to [(\alpha_1,r_1)(\alpha_1,r_l)]
    \end{equation*}
    is injective. Therefore, 
    \begin{equation*}
        [(0,R_2)(\e,R_2)] \to [(\e,R_2)(\alpha,r_1)]
        \quad\text{ and }\quad
        [(0,R_u)(\e,R_u)] \to [(\e,R_u)(\alpha,r_1)]
    \end{equation*}
    are isomorphisms, and so
    \begin{equation*}
        R_2, R_u \in A_{x}(\e | r=r_l, \alpha = \alpha_1).
    \end{equation*}

    Part \eqref{lem:A_4} follows from the above results and the fact
    that intersection of intervals is an interval.
\end{proof}

\begin{proof}[Proof of Lemma \ref{lem:E_prop}]\leavevmode

    The first part follows immediately from the definitions and  the
    fact that\\
    \centerline{$A_{x}([0,\e_1]\,|\, \alpha\leq\beta_1) \supseteq
    A_{x}([0,\e_2]\,|\, \alpha\leq\beta_2)$.}

    For part two, note that
    $$
    0\leq \dd_{x}(\e,\beta)\leq
    \sup_{\alpha'\in[\alpha_{x}^{l}(\e), \alpha_{x}^{u}(\e)]}{
        d_H(A_{x}(0\,|\, \alpha=\alpha'), A_{x}(\e\,|\,
        \alpha=\alpha'))}\to 0\quad\text{as}\quad\e\to 0.
    $$
    The other limit follows from the fact that $\alpha$-sections of
    $A_{x}(0|\alpha\in[0,\beta])$ are right isosceles triangles (in the
    $(R,r)$-plane) whose lower legs descend to the $R$-axis as $\beta$ $\to
    0$ (see proofs of Lemmas \ref{lem:A_0_X}, \ref{lem:A_0} for details).

    Part three now follows from parts one and two. Indeed, weak
    seemliness implies that
    \begin{equation*}
        A_{x}([0,\e]\,|\, \alpha \leq \beta)
    \end{equation*}
    has non-empty interior for small enough $\e$ and $\beta\in[\alpha_{x}^{l}(\e),
    \alpha_{x}^{u}(\e)]$. Hence,
    \begin{equation*}
        \tau_{x}(\e,\beta)>0.
    \end{equation*}
    Then part two implies that by reducing $\e$ we can achieve
    \begin{equation*}
        \tau_{x}(\e',\beta)>\gamma,
    \end{equation*}
    where
    \begin{equation*}
        \e' = \alpha_{x}^{l}(g(0,\e)),\quad
        \beta = g(\e',\e),\quad
        \gamma=f(0,\e)+f(\e',\e).
    \end{equation*}
    Hence, $\e\in E_{x}(f,g)$.

    The claim that $E_{x}(f,g)$ is an interval follows from the
    monotonicity of $\tau_{x}(\e,\beta)$, since it implies that if
    $\e_2\in E_{x}(f,g)$ and $\e_1\leq\e_2$, then $\e_1\in E_{x}(f,g)$.
\end{proof}

\begin{proof}[Proof of Lemma \ref{lem:A_0_X}]\leavevmode
    We shall prove the statement of the lemma for any compact and
    non-empty $L\subseteq X$, $X\in\st$.  It is enough to show that
    $\alpha$-sections $A_{L}(0\,|\, \alpha=\e)\neq\emptyset$ for all
    sufficiently small $\e>0$. The rest of the statement follows from
    Lemma \ref{lem:A_0} and the fact that any intersection of right
    isosceles triangles with legs parallel to the axes and the
    hypotenuse lying on the diagonal is another such triangle.
    
    Let
    \begin{align*}
        W_1(x) &=
        \{\rho>0\,|\,\forall \rho'\in(0,\rho]\; \h(K, K-B_{\rho'}(x))
    \approx \h(K, K-\{x\})\},\\
        W_2(x) &=
        \{\rho>0\,|\,\forall \rho'\in(0,\rho]\; D_{\rho'}(x)\text{ is a
    topological ball}\},\\
        W_3(x) &=
        \left\{
            \begin{aligned}
                &\{\rho>0\,|\,\forall \rho'\in(0,\rho]\; S_{\rho'}(x)\pitchfork K\},&
            \quad K \text{ is a Whitney stratified set}\\
                &\X,&\quad \text{otherwise.}
            \end{aligned}
        \right.
    \end{align*}
    Define
    \begin{equation*}
        \bar{\rho}(x) = \sup{(W_1(x)\cap W_2(x)\cap W_3(x))}.
    \end{equation*}
    The fact that $\X$ locally strongly convex, the definition of
    homology stratification, and properties of
    Whitney stratified sets imply $\bar{\rho}(x)>0$ for any $x\in K$.
    We shall show that $\bar{\rho}(L)$ $=$ $\inf_{x\in L}{\bar{\rho}(x)}>0$.
    Suppose the opposite. Then we can find a
    sequence $x_n\in L$ such that $\bar{\rho}(x_n)\to 0$.
    Due to compactness of $L$, we can assume without loss of
    generality that $x_n$ is convergent. Let
    $\lim_{n\to\infty}{x_n} = \hat{x}\in L$. For any
    $\rho$ $\in$ $(0,\bar{\rho}(\hat{x}))$ we have $x_n$ $\in$ $B_{\rho}(\hat{x})$
    for all sufficiently large $n$. Take any
    $\rho'$ $\in$ $(0,\rho-d(\hat{x},x_n))$.
    Then we have induced homomorphisms:
    $$
    \h(K, K-B_{\rho}(\hat{x})) \to \h(K, K-B_{\rho'}(x_n)) \to
    \h(K, K-\{x_n\}).
    $$
    Conditions for homology stratification imply that if 
    $\rho$ is sufficiently small then all of the above homomorphisms are
    isomorphisms. In particular,
    \begin{equation*}
        \h(K, K-B_{\rho'}(x_n)) \to \h(K, K-\{x_n\})
    \end{equation*}
    is an isomorphism. But we can have $\rho'>\bar{\rho}(x_n)$ for large
    $n$. Contradiction.
    It follows that $A_{L}(0)\neq\emptyset$ since contains triples of the form
    $(\rho, \rho',0)$, where $\rho,\rho'$ $\in$ $(0,\bar{\rho}(L))$, $\rho\geq \rho'$.

    Let $U$ be a small enough neighborhood of $K$ so that there is a
    retraction $\pi:U\to K$. 
    Take $\rho_u>0$, $\beta>0$. For any $x\in K$ and any $\rho\in[0,\rho_u]$,
    the preimage
    \begin{equation*}
        U_{\rho}(x) = \pi^{-1}(K-D_{\rho}(x))
    \end{equation*}
    is an open set containing $K-B_{\rho+\beta}(x)$. Due to compactness
    of the latter, we can find
    $\e(x,\rho,\beta)>0$ such that
    \begin{equation*}
        D_{\e(x,\rho,\beta)}(K) - B_{\rho+\beta}(x) \subseteq U_{\rho}(x).
    \end{equation*}
    We claim that $\e(x,\rho,\beta)$ can be chosen so that
    \begin{equation*}
        \bar{\e}(\beta) = \inf\{\e(x,\rho,\beta)\,|\, x\in L,\rho\in[0,\rho_u]\}>0.
    \end{equation*}
    Suppose the opposite. Then we can find $x_n\in K$, $\rho_n\in[0,\rho_u]$, such that
    $\e(x_n,\rho_n,\beta)$ $\to$ $0$. Due to compactness of $K$ and
    $[0,\rho_u]$, we can assume $x_n$ $\to$ $\hat{x}$ $\in$ $K$,
    $\rho_n$ $\to$ $\hat{\rho}$ $\in$ $[0,\rho_u]$. Taking $\alpha>0$ sufficiently small,
    we get
    \begin{equation*}
        U_{\hat{\rho}+\alpha}(\hat{x}) \supseteq
        D_{\e(\hat{x}, \hat{\rho},\beta)}(K) - B_{\hat{\rho}+\beta-\alpha}(\hat{x}).
    \end{equation*}
    For all sufficiently large $n$, we have
    \begin{equation*}
        \e(\hat{x}, \hat{\rho},\beta) > \e(x_n,\rho_n,\beta),\quad
        B_{\hat{\rho}+\beta-\alpha}(\hat{x}) \subseteq
        B_{\hat{\rho}+\beta}(x_n),\quad
        D_{\hat{\rho}+\alpha}(\hat{x}) \supseteq D_{\hat{\rho}}(x_n).
    \end{equation*}
    Hence,
    \begin{equation*}
    U_{\hat{\rho}}(x_n)\supseteq
    U_{\hat{\rho}+\alpha}(\hat{x})\supseteq
    D_{\e(\hat{x},\hat{\rho},\beta)}(K)-B_{\hat{\rho}+\beta-\alpha}(\hat{x})\supseteq
    D_{\e(\hat{x},\hat{\rho},\beta)}(K)-B_{\hat{\rho}+\beta}(x_n).
    \end{equation*}
    This shows that
    $\e(x_n,\rho_n,\beta)$ could
    have been chosen as large as $\e(\hat{x},\hat{\rho},\beta)$. Contradiction.
    
    Now, suppose $\rho_u$ $\in$ $(0,\bar{\rho}(L))$,
    $\beta$ $\in$ $(0,\bar{\rho}(L)-\rho_u)$. Take
    \begin{equation*}
        \rho \in [0,\rho_u],\quad
        \rho' \in [\rho+\beta,\bar{\rho}(L)),\quad
        \e\in(0,\bar{\e}(\beta)),\quad
        x\in L.
    \end{equation*}
    The image $\pi(D_{\e}(K)-B_{\rho'}(x))$ is a compact set
    inside $K-D_{\rho}(x)$. Hence,
    \begin{equation*}
        \pi(D_{\e}(K)-B_{\rho'}(x))\subseteq K-B_{\rho''}(x)
    \end{equation*}
    for some $\rho''>\rho$. Consider
    \begin{equation*}
    \begin{tikzcd}[column sep=small]
        \h(K, K-B_{\rho'}(x))\ar[r, "i_*"] & \h(D_{\e}(K),
        D_{\e}(K)-B_{\rho'}(x))\ar[r, "\pi_*"]&
    \h(K, K-B_{\rho''}(x))
    \end{tikzcd}
    \end{equation*}
    Since $\pi_*\circ i_*$ is an isomorphism, $i_*$ must be injective,
    yielding $(\rho',\rho',\e)\in A_{L}(0)$. Hence,
    $A_{L}(0\,|\,\alpha=\e)\neq\emptyset$.
    
    It is worth pointing out
    that the aforementioned argument implies that no mater how small
    $\rho'>0$ is, we can find $\e>0$ such that $(\rho',\rho',\e)\in
    A_{L}(0)$.

\end{proof}

\begin{proof}[Proof of Lemma \ref{lem:seemly_K}]\leavevmode
    For simplicity, we shall use sub-indexes $ij$ instead of
    $X_{ij}^{\bm{w}}$ and $\bm{w}$ instead of
    $K^{\bm{w}}$.
    
    The definition of seemliness gives us, for each
    $X_{ij}^{\bm{w}}$, functions
    $\alpha_{ij}^{l}$, $\alpha_{ij}^{u}:$ $[0,\e_{ij}]\to\R_+$.
    Let $\e_{\bm{w}}$ $=$ $\min\{\e_{ij}\}>0$. Then we have
    functions
    \begin{equation*}
        \alpha_{\bm{w}}^{l} = \max\{\alpha_{ij}^{l}\}
        \quad\text{ and }\quad
        \bar{\alpha}_{\bm{w}}^{u} = \min\{\alpha_{ij}^{u}\}
    \end{equation*}
    defined on $[0,\e_{\bm{w}}]$.

    It follows from the proof of Lemma \ref{lem:A_0_X} that
    $\alpha$-sections of $A_{\bm{w}}(0)$, which is the intersection of all
    $A_{ij}(0)$, are non-empty for sufficiently small $\alpha$, say,
    $\alpha\in[0,\bar{\alpha}]$, and have the structure described in the
    lemma. We may reduce $\e_{\bm{w}}$, if necessary, to make sure that 
    $\alpha_{\bm{w}}^{l}(\e_{\bm{w}})\leq\bar{\alpha}$, and define
    $\alpha_{\bm{w}}^{u}$ $=$ $\min\{\bar{\alpha}_{\bm{w}}^{u},
    \bar{\alpha}\}$.

    The properties of $\alpha$-sections from Lemma
    \ref{lem:A} imply that if
    \begin{equation*}
        \begin{aligned}
            &D_H(A_{ij}(0\,|\,\alpha=\alpha'),
            A_{ij}(\e\,|\,\alpha=\alpha')) < a,\\
            &D_H(A_{i'j'}(0\,|\,\alpha=\alpha'),
            A_{i'j'}(\e\,|\,\alpha=\alpha')) < b,
        \end{aligned}
    \end{equation*}
    then
    \begin{equation*}
        D_H(A_{ij}(0\,|\,\alpha=\alpha')\cap A_{i'j'}(0\,|\,\alpha=\alpha'),
        A_{ij}(\e\,|\,\alpha=\alpha')\cap A_{i'j'}(\e\,|\,\alpha=\alpha')) <
        \max\{a,b\}.
    \end{equation*}
    It then follows that
    \begin{equation*}
        \sup_{\alpha'\in[\alpha_{\bm{w}}^{l}(\e),
        \alpha_{\bm{w}}^{u}(\e)]}{d_H(A_{\bm{w}}(0\,|\, \alpha=\alpha'), A_{\bm{w}}(\e\,|\,
        \alpha=\alpha'))}\to 0\text{ as } \e\to 0,
    \end{equation*}
    which proves the lemma.

\end{proof}

\begin{acknowledgements}
This work has been supported by the National Science Foundation grant
DMS-1622370.
\end{acknowledgements}

\bibliographystyle{plain}
\bibliography{refs}

\begin{thebibliography}{10}

\bibitem{bekka1991}
K.~Bekka.
\newblock C-r{\'e}gularit{\'e} et trivialit{\'e} topologique.
\newblock In David Mond and James Montaldi, editors, {\em Singularity Theory
  and its Applications}, pages 42--62. Springer Berlin Heidelberg, 1991.

\bibitem{bendich.etal2007}
P.~Bendich, D.~Cohen-Steiner, H.~Edelsbrunner, J.~Harer, and D.~Morozov.
\newblock Inferring local homology from sampled stratified spaces.
\newblock In {\em 48th Annual IEEE Symposium on Foundations of Computer Science
  (FOCS'07)}, pages 536--546, 2007.

\bibitem{bendich.etal2010}
Paul Bendich, Sayan Mukherjee, and Bei Wang.
\newblock Stratification learning through homology inference.
\newblock In {\em AAAI Fall Symposium Series}, pages 10--17, 2010.

\bibitem{bredon2013}
Glen~E Bredon.
\newblock {\em Topology and geometry}, volume 139 of {\em Graduate Texts in
  Mathematics}.
\newblock Springer-Verlag New York, 1997.

\bibitem{burago.etal2001}
Dmitri Burago, Yuri Burago, and Sergei Ivanov.
\newblock {\em A Course in Metric Geometry}, volume~33 of {\em Graduate Studies
  in Mathematics}.
\newblock American Mathematical Society, 2001.

\bibitem{carlsson2008local}
Gunnar Carlsson, Tigran Ishkhanov, Vin De~Silva, and Afra Zomorodian.
\newblock On the local behavior of spaces of natural images.
\newblock {\em International journal of computer vision}, 76(1):1--12, 2008.

\bibitem{chan2013topology}
Joseph~Minhow Chan, Gunnar Carlsson, and Raul Rabadan.
\newblock Topology of viral evolution.
\newblock {\em Proceedings of the National Academy of Sciences},
  110(46):18566--18571, 2013.

\bibitem{chazal_oudot2008}
Fr{\'e}d{\'e}ric Chazal and Steve~Yann Oudot.
\newblock Towards persistence-based reconstruction in euclidean spaces.
\newblock In {\em Proceedings of the Twenty-fourth Annual Symposium on
  Computational Geometry}, SCG '08, pages 232--241, 2008.

\bibitem{cheniot1972sections}
D.~Cheniot.
\newblock Sur les sections transversales d'un ensemble stratifi{\'e}.
\newblock {\em CR Acad. Sci. Paris S{\'e}rie AB}, 275:915--916, 1972.

\bibitem{cohen-steiner.etal2007}
David Cohen-Steiner, Herbert Edelsbrunner, and John Harer.
\newblock Stability of persistence diagrams.
\newblock {\em Discrete {\&} Computational Geometry}, 37(1):103--120, 2007.

\bibitem{de2007coverage}
Vin de~Silva and Robert Ghrist.
\newblock Coverage in sensor networks via persistent homology.
\newblock {\em Algebraic \& Geometric Topology}, 7(1):339--358, 2007.

\bibitem{dey.etal2014}
Tamal~K Dey, Fengtao Fan, and Yusu Wang.
\newblock Dimension detection with local homology.
\newblock In {\em Proceedings of the 26th Canadian Conference on Computational
  Geometry}, pages 273--279, 2014.

\bibitem{dold1995}
Albrecht Dold.
\newblock {\em Lectures on algebraic topology}.
\newblock Classics in Mathematics. Springer-Verlag Berlin Heidelberg, 1995.

\bibitem{dugundji1967}
J.~Dugundji.
\newblock Maps into nerves of closed coverings.
\newblock {\em Annali della Scuola Normale Superiore di Pisa - Classe di
  Scienze}, 21(2):121--136, 1967.

\bibitem{edelsbrunner2002}
Herbert Edelsbrunner, David Letscher, and Afra Zomorodian.
\newblock Topological persistence and simplification.
\newblock {\em Discrete {\&} Computational Geometry}, 28(4):511--533, 2002.

\bibitem{eilenberg.steenrod1952}
Samuel Eilenberg and Norman Steenrod.
\newblock {\em Foundations of algebraic topology}.
\newblock Princeton University Press, 1952.

\bibitem{ghrist2008barcodes}
Robert Ghrist.
\newblock Barcodes: the persistent topology of data.
\newblock {\em Bulletin of the American Mathematical Society}, 45(1):61--75,
  2008.

\bibitem{goresky1988strat}
Mark Goresky and Robert MacPherson.
\newblock {\em Stratified morse theory}.
\newblock Springer, 1988.

\bibitem{grove1993}
Karsten Grove.
\newblock Critical point theory for distance functions.
\newblock In {\em Proceedings of Symposia in Pure Mathematics}, volume~54,
  pages 357--385, 1993.

\bibitem{horak2009persistent}
Danijela Horak, Slobodan Maleti{\'c}, and Milan Rajkovi{\'c}.
\newblock Persistent homology of complex networks.
\newblock {\em Journal of Statistical Mechanics: Theory and Experiment},
  2009(03):P03034, 2009.

\bibitem{mather2012notes}
John Mather.
\newblock Notes on topological stability.
\newblock {\em Bulletin of the American Mathematical Society}, 49(4):475--506,
  2012.

\bibitem{niyogi.etal2008}
Partha Niyogi, Stephen Smale, and Shmuel Weinberger.
\newblock Finding the homology of submanifolds with high confidence from random
  samples.
\newblock {\em Discrete {\&} Computational Geometry}, 39(1):419--441, 2008.

\bibitem{niyogi2011topological}
Partha Niyogi, Stephen Smale, and Shmuel Weinberger.
\newblock A topological view of unsupervised learning from noisy data.
\newblock {\em SIAM Journal on Computing}, 40(3):646--663, 2011.

\bibitem{petersen2006}
Peter Petersen.
\newblock {\em Riemannian geometry}, volume 171 of {\em Graduate Texts in
  Mathematics}.
\newblock Springer, 2006.

\bibitem{pflaum2001}
Markus Pflaum.
\newblock {\em Analytic and geometric study of stratified spaces: contributions
  to analytic and geometric aspects}, volume 1768 of {\em Lecture Notes in
  Mathematics}.
\newblock Springer-Verlag Berlin Heidelberg, 2001.

\bibitem{pflaum2017}
Markus~J Pflaum and Graeme Wilkin.
\newblock Equivariant control data and neighborhood deformation retractions.
\newblock {\em arXiv preprint arXiv:1706.09539}, 2017.

\bibitem{quinn1988}
Frank Quinn.
\newblock Homotopically stratified sets.
\newblock {\em Journal of the American Mathematical Society}, 1(2):441--499,
  1988.

\bibitem{siebenmann1972}
L.~C. Siebenmann.
\newblock Deformation of homeomorphisms on stratified sets.
\newblock {\em Commentarii Mathematici Helvetici}, 47(1):123--163, 1972.

\bibitem{skraba.wang2014}
Primoz Skraba and Bei Wang.
\newblock Approximating local homology from samples.
\newblock In {\em Proceedings of the twenty-fifth annual ACM-SIAM symposium on
  Discrete algorithms}, pages 174--192, 2014.

\bibitem{zomorodian2005}
Afra Zomorodian and Gunnar Carlsson.
\newblock Computing persistent homology.
\newblock {\em Discrete {\&} Computational Geometry}, 33(2):249--274, 2005.

\end{thebibliography}

\end{document}